%% file: Counting-Orbits-nmd.tex
\documentclass[a4,11pt]{nmd/article}
\usepackage{graphicx,color}
\usepackage{epsfig,manfnt}
\usepackage{amsfonts}
\usepackage{amssymb}
\usepackage{amsmath}
\usepackage{amsthm}
\usepackage{latexsym}
\usepackage{amscd}
\usepackage{flafter}
\usepackage{epsf}
\usepackage{epstopdf}
\usepackage{inputenc}
\usepackage{stmaryrd}
\usepackage{tikz,tikz-cd}
\usepackage{mathrsfs}
\usepackage{lipsum}
\input{commands-file-nmd}

\begin{document}
\title[Counting periodic orbits of vector fields over smooth closed manifolds]
{Counting periodic orbits of vector fields over smooth closed manifolds }%
\author{Eaman Eftekhary}%
\address{School of Mathematics, Institute for Research in 
Fundamental Sciences (IPM), P. O. Box 19395-5746, Tehran, Iran}%
\email{eaman@ipm.ir}
\begin{abstract}
We address the problem of counting periodic orbits of vector fields on smooth closed manifolds. The space of non-constant periodic orbits is enlarged to a complete space by adding the {\emph{ghost orbits}}, which are decorations of the zeros of vector fields. Associated with any compact and open subset $\Gamma$ of the moduli space of periodic and ghost orbits, we define an integer weight. When the vector field moves along a path, and $\Gamma$ deforms in a compact and open family, we show that the weight function stays constant. We also give a number of examples and computations, which illustrate the applications of our main theorem.  
\end{abstract}

\maketitle
\tableofcontents
\section{Introduction}
Let $\Ycal$ denote the space of all $C^1$ vector fields on a smooth closed manifold $M$ of dimension $n$, and $\Ycal^\circ\subset\Ycal$ denote the space of non-singular (i.e. nowhere vanishing) vector fields. The Poincar\'e-Hopf index theorem implies that   $\Ycal^\circ$ is non-empty if and only if $\chi(M)=0$. Periodic orbits are among the objects which make the dynamics of the flow associated with vector fields in $\Ycal^\circ$ interesting and challenging. When $M$ is a $3$-manifold, it follows from Kuperberg's results on the Seifert conjecture that each connected component of $\Ycal^\circ$ includes a volume-preserving vector field without periodic orbits \cite{Kuperberg,Kuperberg-2} (see also \cite{Ghys}). The search for vector fields with least number of periodic orbits may be restricted  to specific subsets of $\Ycal^\circ$. Let $\Ycal^{\mathrm{c}}\subset \Ycal^{\mathrm{h}}\subset\Ycal^\circ$ denote the subsets consisting of the Reeb vector fields associated with contact and stable Hamiltonian structures on $M$, respectively. The Weinstein conjecture for contact $3$-manifolds, proved by Taubes and by Cristofaro-Gardiner and Hutchings for contact structures implies that  every $Y\in \Ycal^{\mathrm{c}}$ admits at least two periodic orbit \cite{Taubes-2,Gardiner-Hutchings}. Moreover, Hutchings and Taubes show that every $Y\in\Ycal^{\mathrm{h}}$ admits at least one periodic orbit provided that $M$ is not a torus bundle over the circle \cite{Hutchings-Taubes}. \\  

A smooth closed manifold $N$ may be viewed as the zero section in its tangent bundle $TN$. The quotient  $M$  of $TN\setminus N$ by the scaling action of $\R^+$ may be viewed as the unit tangent bundle of $N$ for every Riemannian metric $g$ on $N$. The metric $g$ determines a vector field $Y_g\in\Ycal^\circ$ which generates the geodesic flow associated with $g$ on $M$. Periodic orbits of $Y_g$ correspond to closed $g$-geodesics. Such vector fields form a subset $\Ycal^{\mathrm{r}}\subset\Ycal^\circ$. In each free homotopy class of closed loops on  $N$, there is at least one closed $g$-geodesic (and exactly one closed $g$-geodesic if $g$ is negatively curved). Margulis shows that the asymptotics of the function counting closed geodesics below a given length is related to the topological entropy of the metric over manifolds with constant negative curvature \cite{Margulis}. This correspondence was  generalized to manifolds with variable negative curvature by Philips and Sarnak \cite{PS} (see also \cite{Anan} and  \cite{PS-1}).\\

Let $M=M_\phi$ denote the suspension any isotopy class in $\mathrm{Diff}(N)$, which consists of self-diffeomorphisms of a  smooth closed manifold $N$.  Associated with every self-diffeomorphism $f$ representing $\phi$  is (the conjugacy class of) a vector field $Y_f$ on $M$. Such vector fields form a subset $\Ycal^\phi\subset\Ycal^\circ$, while periodic orbits of $Y_f$ are in correspondence with periodic orbits of $f$. The Lefshetz-Hopf fixed point theorem for iterations of $f$ gives a formula for the {\emph{weighted}} count of periodic orbits of $f$, which remains constant as $f$ moves in $\phi$. Nevertheless, the number of periodic orbits can be higher  if particular subsets of $\mathrm{Diff}(N)$ are considered. For instance, if $\mathrm{Ham}(N,\omega)$ denotes the set of Hamiltonian self-diffeomorphisms of a symplectic manifold $(N,\omega)$, the Arnold conjecture predicts that the number of fixed points of every $f\in\mathrm{Ham}(N,\omega)$  is bounded below by the number of critical points of smooth functions on $N$, and thus  by the Lusternik–Schnirelmann category of $N$ \cite{Takens}. The sum of the betti numbers of $N$ is already proved to be a lower bound in this context \cite{Liu-Tian, Fukaya-Ono, Ruan}, by refining the approach of Floer in the monotone case \cite{Floer}.\\ 

Such results investigate the vector fields in a given isotopy class, with the least number of periodic orbits. As in the Lefshetz-Hopf formula for $Y\in\Ycal^\phi$, one may alternatively investigate the weighted counts of periodic orbits for the vector fields $Y\in\Ycal^\circ$ which stay constant as $Y$ moves in a family. The results on the Seifert conjecture illustrates some of the potential difficulties, as non-singular vector fields with interesting periodic orbits are isotopic to vector fields without any periodic orbits.  Infinite families of periodic orbits for $Y\in\Ycal^\circ$ correspond to discrete sets of periodic orbits when $Y$  is perturbed to a close-by  generic vector field. Nevertheless, a major problem is that different perturbations of $Y$ give structurally different discrete sets. When $M$ is the unit tangent bundle of $N$ and $Y\in\Ycal^{\mathrm{r}}$, the problem of counting closed geodesics is discussed in \cite{Ef-p}. Here,  we extend the results of \cite{Ef-p} to vector fields in $\Ycal$ on arbitrary smooth closed manifolds.\\

 Given $Y\in\Ycal$, every (non-constant) periodic $Y$-orbit is the degree-$d$ cover of an embedded periodic orbit for some $d\in\Z^+$, which is called the {\emph{degree}} of that periodic orbit.  
Let $\EOrbits_Y=\cup_d \EOrbits_Y^d$ denote the space of (non-constant) periodic $Y$-orbits, where $\EOrbits_Y^d$ denotes the subset of degree $d$ orbits. The space $\EOrbits$ of periodic orbits is then a fiber space over $\Ycal$, where the fiber over $Y\in\Ycal$ is given by $\EOrbits_Y$. Note that  $\EOrbits$   admits a rotation action of  $S^1$,  which preserves the fibers. The quotient $\EOrbits/S^1$ is called the {\emph{moduli space}} of periodic orbits. We denote the projection map from $\EOrbits$ to $\Ycal$ by $p:\EOrbits\ra \Ycal$. If $\vY:[0,1]\ra \Ycal$ is a $C^1$ path of vector fields connecting $Y_0,Y_1\in\Ycal$, the moduli space $\EOrbits_\vY$ is also defined as a fiber space over $[0,1]$ where the fiber over $t\in[0,1]$ is $\EOrbits_{\vY(t)}$.
 
\begin{thm}\label{thm:intro-nonsingular-case}
Associated with every $Y\in\Ycal$ and every compact and open subset $\Gamma$ of $\EOrbits_Y/S^1$ is a weight $n(\Gamma)\in\Z$. This weight function satisfies the following two properties:
\begin{itemize}
\item If $\Gamma$ consists of a periodic orbit of degree $1$ with linearized holonomy $\Fmap:\R^{n-1}\ra \R^{n-1}$, where $\Fmap-Id$ is non-singular, then $n(\Gamma)=\sgn(\det(\Fmap-Id))$.
\item If $\vY:[0,1]\ra \Ycal$ is a path of vector fields connecting $Y_0,Y_1\in\Ycal$ and $\vGamma\subset\EOrbits_\vY/S^1$ is a compact and open subset which intersects $\EOrbits_{\vY(t)}$ in $\Gamma_t$, we have $n(\Gamma_0)=n(\Gamma_1)$.
\end{itemize}
\end{thm} 

When $\Gamma$ has the structure of a closed $k$-dimensional manifold, it is sometimes possible to compute the weight function $n(\Gamma)$ without difficulty.

\begin{prop}\label{prop:computation-intro}
Suppose that $M$ and $\Ycal$ are before, $Y\in\Ycal$ and $\Gamma$ is a compact and open subset of $\EOrbits^1_Y/S^1$, which has the structure of a closed $k$-dimensional manifold. Let $d\star\Gamma$ denote the union of periodic orbits which are degree-$d$ covers of the orbits in $\Gamma$. Suppose that for every periodic orbit in $\Gamma$, the linearization $\Fmap$ of the holonomy map has $1$ as an eigenvalue with multiplicity $k$, while it has no other eigenvalues which are roots of unity.  If $m_1(\Fmap)$ and $m_2(\Fmap)$ denote the number of real eigenvalues of $\Fmap$ in $(-\infty,1)$ and $(-1,1)$, respectively, then $m_(\Fmap)$ and $m_2(\Fmap)$ remain constant on $\Gamma$ and we have  
\begin{align*}
n(\Gamma)=(-1)^{m_1(\Fmap)}\cdot \chi(\Gamma)\quad \text{and}\quad n(2\star \Gamma)=\frac{(-1)^{m_2(\Fmap)}-(-1)^{m_1(\Fmap)}}{2}\cdot \chi(\Gamma),
\end{align*}
while $n(d\star\Gamma)=0$ for all integers $d>2$. Here $\chi(\Gamma)$ denotes the Euler characteristic of $\Gamma$. 
\end{prop}

\begin{examp}\label{ex:intro-1}
The space of periodic orbits of degree $1$ for the Hopf vector field on $S^3$ may be identified with $S^3$, as the orbit of each point is periodic of period $2\pi$, and the holonomy map associated with each periodic orbit is the identity. The quotient of the space of periodic orbits by the rotation action of $S^1$ is $\Gamma=S^2$, which is compact and open in the moduli space $\EOrbits/S^1$. Proposition~\ref{prop:computation-intro} implies that $n(\Gamma)=2$. If $\vY$ is a path of non-singular vector fields on $S^3$ which connects the Hopf vector field to an isotopic non-singular vector field without periodic orbits, it follows that for every period $s>2\pi$, there is a time $t\in[0,1]$ such that $\vY(t)$ has a periodic orbit of period $s$. This is true, since otherwise we may take $\vGamma$ to be the subset of $\EOrbits_{vY}$ with period less than $s$, and 
\[2=n(\Gamma)=n(\Gamma_0)=n(\Gamma_1)=n(\emptyset)=0.\]  
This contradiction implies the claim.
\qed
\end{examp}

The main weakness of the above theorem is that $\EOrbits/S^1$ is not complete. A Cauchy sequence in $\EOrbits/S^1$ may converge to a constant orbit, which is not in $\EOrbits/S^1$. This is of course irrelevant in dealing with non-singular vector fields. Nevertheless, adding constant orbits to the moduli space of periodic orbits, without further consideration, makes significant damage to the invariance claim of Theorem~\ref{thm:intro-nonsingular-case}. The more detailed explanation of the obstacles and stating the stronger theorem requires the introduction of some notation, and the replacement of {\emph{ghost orbits}} for constant orbits.\\ 

A {\emph{ghost}} orbit is a $4$-tuple $\corbit=(s,Y,x,P)$, where $Y\in\Ycal$, $x$ is a zero of $Y$ and $P\subset T_xM$ is a plane which is invariant under $d_xY:T_xM\ra T_xM$. We further require that the restriction $d_xY|_P$ has a pair of eigenvalues $\lambda,\bar\lambda\in\C\setminus \R$ with $\tracet(\corbit)=\lambda+\bar\lambda< 0$ and the {\emph{period}} $s=s(\corbit)\in\R^+$ is such that $s\cdot\mathrm{Im}(\lambda)\in 2\pi\Z^+$. If $s\cdot\mathrm{Im}(\lambda)=2\pi d$, the positive integer $d$ is called the {\emph{degree}} of $\corbit$. A $4$-tuple  $\corbit=(s,Y,x,P)$ is called  a {\emph{boundary}} orbit if $\tracet(\corbit)=0$. The spaces of ghost orbits and boundary orbits are denoted by 
\[\COrbits=\coprod_d\COrbits^d\quad\quad\text{and}\quad\quad\Bcal=\coprod_d\Bcal^d,\] 
respectively, where $\COrbits^d\subset\COrbits$ and $\Bcal^d\subset\Bcal$ are the connected subspaces  consisting of degree-$d$ objects. Both $\COrbits$ and $\Bcal$ inherit their metrics as subspaces of $\R^+\times\Ycal\times G$, where $G=G_M$ is the Grassmannian bundle of $2$-planes in $TM$. The union of $\COrbits$ with $\Bcal$ is sometimes denoted by $\oCOrbits$, which is a $C^1$ Banach manifold with boundary $\Bcal$ (also a $C^1$ Banach manifold). Periodic and ghost orbits have the boundary orbits as their common boundary, as the following theorem implies.

\begin{thm}\label{thm:intro-gluing}
There is a complete metric space $\OModuli=\cup_{d\in\Z^+}\OModuli^d$,  which includes $\Bcal$ as a subspace, while $\OModuli\setminus\Bcal$ is the disjoint union of $\EOrbits/S^1$ and $\COrbits$. There is a projection map $\pi:\OModuli\ra \Ycal$ which restricts to give the Fredholm projection maps $q:\COrbits\ra \Ycal$ and $p^d:\EOrbits^d/S^1\ra \Ycal$ of index zero. Each subspace $\OModuli^d$ is a $C^1$ Banach manifold, which includes $\Bcal^d$ as a codimension $1$ submanifold, with $\OModuli^d\setminus \Bcal^d=\COrbits^d\amalg(\EOrbits^d/S^1)$. 
\end{thm}  

For $Y\in \Ycal$, let $\OModuli_Y=\pi^{-1}(Y)$, $\Bcal_Y=\OModuli_Y\cap\Bcal$ and $\COrbits_Y=\OModuli_Y\cap\COrbits$. In view of Theorem~\ref{thm:intro-gluing}, a sequence of periodic orbits may converge to a boundary orbit, which is also the limit point of a sequence of ghost orbits. Instead of extending the definition of the weight function to the open and compact subsets of $\oEOrbits_Y/S^1=\Bcal_Y\cup(\EOrbits_Y/S^1)$, it is thus more natural to consider open and compact subset of $\OModuli_Y$. The following theorem is the strong version of Theorem~\ref{thm:intro-nonsingular-case}.

\begin{thm}\label{thm:intro-general-case}
Associated with every $Y\in\Ycal$ and every compact and open subset $\Gamma$ of $\OModuli_Y$ is a weight $n(\Gamma)\in\Z$. This weight function satisfies the following two properties:
\begin{itemize}
\item If $\Gamma$ consists of a periodic orbit of degree $1$ with linearized holonomy $\Fmap:\R^{n-1}\ra \R^{n-1}$, where $\Fmap-Id$ is non-singular, then $n(\Gamma)=\sgn(\det(\Fmap-Id))$.
\item If $\Gamma$ consists of a single ghost orbit $\corbit=(s,Y,x,P)$ and $d_xY$ has no eigenvalues on $i\R$ and the complex eigenspace of one of its eigenvalues is $1$-dimensional, then $n(\Gamma)=-\sgn(\det(d_xY))$.
\item If $\vY:[0,1]\ra \Ycal$ is a path of vector fields connecting $Y_0,Y_1\in\Ycal$ and $\vGamma\subset\OModuli$ is a compact and open subset which intersects $\EOrbits_{\vY(t)}$ in $\Gamma_t$, we have $n(\Gamma_0)=n(\Gamma_1)$.
\end{itemize}
\end{thm} 

\begin{examp}\label{ex:intro-2}
In Example~\ref{ex:intro-1}, if $\vY$ is a path of (possibly singular) vector fields on $S^3$ which connects the Hopf vector field to a non-singular vector field without periodic orbits, it follows that for every period $s>2\pi$, there is a time $t\in[0,1]$  such that $\vY(t)$ either has a periodic orbit of period $s$, or a zero $x$ such that $d_xY$ has an eigenvalue with imaginary part equal to $2\pi/s$: otherwise, we take $\vGamma$ to be the union of ghost and periodic orbits of period smaller than $s$ in $\OModuli_\vY$, and arrive at a contradiction, similar to Example~\ref{ex:intro-1}. If $\vY$ is a path of volume preserving vector fields (i.e. the corresponding flows preserve the standard volume form on $S^3$),  all ghost $\vY$-orbits have negative weight. It follows from this observation that for every period $s>2\pi$, there is a time $t\in[0,1]$  such that $\vY(t)$ has a periodic orbit of period $s$. 
\qed
\end{examp}  

The technical tool needed for proving Theorem~\ref{thm:intro-general-case} is the construction the Kuranishi models which describe the local structure of $\OModuli_\vY$ for generic paths of vector fields. Two Kuranishi models play an essential role in this direction. The first model was already encountered in \cite{Ef-p} as the moduli space of closed geodesics was investigated. The second model is only encountered when we deal with singular vector fields. In both cases, the required transversality arguments are presented using the language of holonomy maps and the differential of the vector fields at their zeros, which is a major simplification in comparison with the arguments of \cite{Ef-p}. The structure of the arguments are close to the arguments of \cite{Taubes} and \cite{Ef-rigidity} (see also \cite{White-1,White-2}). A systematic discussion of such arguments, which address the notion of {\emph{super-rigidity}} for generic elements in families of elliptic operators may be found in \cite{DW}.

\section{Moduli spaces of ghost and periodic orbits}\label{sec:moduli}
\subsection{Vector fields and their zeros}
Let $M$ be a smooth closed manifold of dimension $n$ with tangent bundle $\rho=\rho_{TM}:TM\ra M$, and $\rho_G:G=G_M\ra M$ denote the fiber bundle over $M$, where the fiber $G_x$ at $x\in M$ is the Grassmannian $\mathbb{G}(T_xM,2)$ of $2$-planes $P$ in $T_xM$.  We denote the points of $G$ by $(x,P)$ where $x\in M$ and $P\subset T_xM$ are as before. The space of $C^1$ sections of $TM$ is denoted by $\Ycal=\Ycal(M)$. Every vector field $Y\in\Ycal$ determines a flow $F_Y:\R\times M\ra M$, which satisfies $F_Y(0,x)=x$ and $dF_Y/dr=Y\circ F_Y$ (here $r$ is the variable parametrizing $\R$). 
Define 
\[\Ncal=\left\{(Y,x,P)\in\Ycal\times G\ \big|\ Y(x)=0\quad\text{and}\quad  d_xY(P)=P\right\}.
\]
For $\xfrak=(Y,x,P)\in\Ncal$, the trace $\tfrak_\Ncal(\xfrak)$ of the restriction $d_xY:P\ra P$ gives a map $\tfrak_\Ncal:\Ncal\ra \R$.
\begin{lem}\label{lem:blow-up}
 $\Ncal$ is a $C^1$ Banach manifold and the projection map  $q_\Ncal:\Ncal\ra \Ycal$ is Fredholm of index $0$. Moreover, the trace map $\tfrak_\Ncal:\Ncal\ra \R$ is regular. Therefore, $\Ncal_r=\tfrak_\Ncal^{-1}(r)$ is a $C^1$ Banach manifold and the projection map $q_r:\Ncal_r\ra \Ycal$ is Fredholm of index $-1$ for all $r\in\R$. 
\end{lem}
\begin{proof} It is straight-forward to observe that 
\[\Ncal^c=\{(Y,x)\in\Ycal\times M\ \big|\ Y(x)=0\}\]
is a $C^1$ Banach manifold. Let $\rho_c:\Ncal^c\ra M$ denote the projection and $\Ucal$  denote the subset of $ \rho_{c}^*G$ which consists of the triples $(Y,x,P)$ such that $d_xY|_P$ is non-degenerate. Define the bundle map 
\[\Psi:\Ucal\ra \rho_{c}^*G,\quad \Psi(\xfrak):=(Y,x,d_xY(P)),\quad\quad\forall\ \xfrak=(Y,x,P)\in\Ucal.\]  
Note that $\Psi$ intersects the identity map $I:\Ucal\ra \Ucal\subset \rho_{c}^*G$ in  $\Ncal$. It thus suffices to show that the intersection of $\Psi$ with $I$ is transverse. Consider a coordinate chart on $M$ which identifies a neighborhood of $x\in M$ with a neighborhood of $0\in\R^n$. Furthermore, assume that $P$ corresponds to $\R^2\times \{0\}\subset \R^2\times \R^{n-2}=\R^n$. This also gives a trivialization of $G$ over the coordinate chart around $x\in M$. The tangent space of the fiber $G_x$ at $P$ may be described, using the coordinate chart, as follows. Let $e_1,\ldots,e_n$ denote the standard coordinate vectors in $\R^n$. Then every plane $P'$ which is sufficiently close to $P$ is spanned by a pair of vectors of the form $e_1'=e_1+f_1$ and $e_2'=e_2+f_2$, where $f_1,f_2\in\R^{n-2}=\langle e_3,\ldots,e_n\rangle$. Therefore, $T_PG_x$ is identified with 
\[\R^{n-2}\oplus\R^{n-2}=\langle e_3,\ldots,e_n\rangle\oplus\langle e_3,\ldots,e_n\rangle.\]
The vector space $T_{\xfrak}\rho_{c}^*G$   includes $\Ycal_x\times\{0\}\subset\Ycal\oplus T_{(x,P)}G$, where $\Ycal_x$ consists of all $Z\in\Ycal$ with $Z(x)=0$. The restriction of any $Z\in\Ycal_x$ to the coordinate chart is a function from a neighborhood of $0\in\R^n$ to $\R^n$ with $Z(0)=0$. The images of $(Z,0)\in T_{\xfrak}\rho_{c}^*G$ under the differential $dI$ of $I$ and the differential $d\Psi$ of $\Psi$ may be projected over $T_PG_x=\R^{n-2}\oplus\R^{n-2}$ to obtain the vectors 
\[DI(Z,0)=0, \ \ D\Psi(Z,0)=\left(p_{\R^{n-2}}\big(d_0Z(e_1)\big),p_{\R^{n-2}}\big(d_0Z(e_2)\big)\right)\in\R^{n-2}\oplus\R^{n-2},\]
where $p_{\R^{n-2}}$ denotes the standard projection map from $\R^n$ to $\R^{n-2}$ which forgets the first two entries. Changing $Z$ in $\Ycal_x$, the above vector in $\R^{n-2}\oplus\R^{n-2}$ can take any value. This implies that the intersection of $I$ and $\Psi$ in $\Ncal$ is transverse. Hence, $\Ncal$ is a $C^1$ Banach manifold. In order to show the regularity of $\tfrak_\Ncal:\Ncal\ra\R$, it suffices to show that the transversality may be achieved using the vector fields $Z\in\Ycal_x$ such that $\mathrm{tr}(d_xZ)=0$. This is true, since the latter condition does not affect the vectors $p_{\R^{n-2}}(d_0Z(e_j))$, for $j=1,2$. Other claims in the statement are straight-forward. 
\end{proof}

If $\xfrak=(Y,x,P)\in\Ncal$, the restriction of $d_xY$ to $P$ has a pair of eigenvalues $\lambda_1,\lambda_2\in\C$. If both $\lambda_1$ and $\lambda_2$ are real, we set $\lambda(\xfrak)=(\lambda_1+\lambda_2)/2$. Otherwise, precisely one of $\lambda_1$ and $\lambda_2$, denoted by $\lambda(\xfrak)$, is in the upper half plane $\Hbb\subset\C$. Let $\Ncal^\circ\subset\Ncal$ denote the pre-image of $\Hbb\setminus\R$ under $\lambda:\Ncal\ra\Hbb$. For $\corbit=(s,\xfrak)\in\R^+\times\Ncal^\circ$, set  $s(\corbit)=s,\lambda(\corbit)=\lambda(\xfrak)$ and $\tfrak(\corbit)=\tfrak_\Ncal(\xfrak)$. We then define
\[\oCOrbits=\left\{\corbit\in\R^+\times \Ncal^\circ\ \big|\ s(\corbit)\cdot \lambda(\corbit)\in 2\pi\Z^+\ \ \text{and}\ \ \tfrak(\corbit)\leq 0\right\}\quad\quad\text{and}\quad\quad \Bcal=\{\corbit\in\oCOrbits\ \big|\ \tfrak(\corbit)=0\},\]
and set $\COrbits=\oCOrbits\setminus\Bcal$. The points in $\COrbits$ and $\Bcal$ are called the {\emph{ghost}} orbits and the {\emph{boundary}} orbits, respectively. Lemma~\ref{lem:blow-up} implies that $\oCOrbits$ is a $C^1$ Banach manifold with  boundary $\Bcal$. For $\corbit\in\oCOrbits$, the positive integer $s(\corbit)\lambda(\corbit)/2\pi$ is called the degree of $\corbit$ and is denoted by $\deg(\corbit)$. Then 
\[\oCOrbits=\coprod_{d\in\Z^+}\oCOrbits^d,\quad\quad \COrbits=\coprod_{d\in\Z^+}\COrbits^d\quad\quad\text{and}\quad\quad\Bcal=\coprod_{d\in\Z^+}\Bcal^d,\]
where $\oCOrbits^d,\COrbits^d$ and $\Bcal^d$ are the subsets of $\oCOrbits,\COrbits$ and $\Bcal$ (respectively) which consist of orbits of degree $d$.  Let $q:\oCOrbits\ra \Ycal$ denote the projection  map. Given $Y\in\Ycal$, we set 
\begin{align*}
\oCOrbits_Y=q^{-1}(Y)=\coprod_{d\in\Z^+}\oCOrbits^d_Y,\quad \COrbits_Y=\oCOrbits_Y\cap\COrbits=\coprod_{d\in\Z^+}\COrbits^d_Y\quad\text{and}\quad\Bcal_Y=\oCOrbits_Y\cap\Bcal=\coprod_{d\in\Z^+}\Bcal^d_Y.
\end{align*}
Let $\Ycal^\infty\subset\Ycal$ denote the subset which consists of the vector fields $Y\in\Ycal$ so that for every zero $x$ of $Y$, $d_xY:T_xM\ra T_xM$ has no eigenvalues on $i\R$, and no eigenvalue of $d_xY$ has non-trivial multiplicity (i.e. the eigenspace associated with every eigenvalue of the induced map $d_xY:T_xM\otimes_\R\C\ra T_xM\otimes_\R\C$ is $1$-dimensional). It is then easy to see that $\Ycal^\infty$ is an open and dense subset of $\Ycal$ and every $Y\in\Ycal^\infty$ is a regular values for both $q$ and its restriction   $b:\Bcal\ra\Ycal$  to $\Bcal$. 

\begin{defn}\label{def:weight-constant}
 $\corbit=(s,Y,x,P)\in\COrbits$ is called {\emph{super-rigid}} if $d_xY$ has no eigenvalues on $i\R$ and no eigenvalue of $d_xY$ has non-trivial multiplicity. For a super-rigid $\corbit\in\COrbits$ as above, define the {\emph{weight}} $n(\corbit)$  equal to $-\sgn(\det(d_xY))$ if $\deg(\corbit)=1$, and set $n(\corbit)=0$, otherwise.
\end{defn}


\subsection{Periodic orbits of vector fields as a moduli space}
 Let $\oXcal$ denote the space of maps of class $W^{2,2}$ from $S^1=\R/\Z$ to $M$ and $\Xcal$  denote the subspace of immersions. We then set 
 \[\Zcal=\R^+\times\Ycal\times \Xcal\quad\text{and}\quad\oZcal=\R^{+}\times\Ycal\times\oXcal.\] 
 For $\orbit=(s,Y,\gamma)\in\oZcal$, set $s_\orbit=s,Y_\orbit=Y$ and $\gamma_\orbit=\gamma$. $\gamma$ is  regarded as a map from $\R$ to $M$ (parametrized by the variable $\theta$), which has period $1$, i.e. $\gamma(\theta+1)=\gamma(\theta)$. If $\dot\gamma=s\cdot Y\circ \gamma$, then  $\orbit=(s,Y,\gamma)\in\Zcal$ is called a {\emph{closed $Y$-orbit}} of period (or length) $s$, where $\dot{\gamma}=\partial_\theta\gamma$ is the differential of $\gamma$ with respect to $\theta$. For $k\in\{1,2\}$ and $\gamma\in\oXcal$, denote the space of $W^{k,2}$ sections of $\gamma^*TM$ by $A^k(\gamma)=A^k(\gamma^*TM)$. Let  $\Ecal$ denote the  vector bundle over $\oZcal$ with fiber $\Ecal_\orbit= A^1(\gamma_\orbit)$ at  $\orbit$. Define the section $\Phi:\oZcal\ra \Ecal$ of the  bundle $\pi_\Ecal:\Ecal\ra \oZcal$ and the spaces of periodic orbits by 
\begin{align*}
&\Phi(\orbit)=\dgam_\orbit-{s_\orbit}\cdot Y_\orbit\circ\gamma_\orbit\in\Ecal_{\orbit}\quad\text{and}\quad  \EOrbits=\left\{\orbit\in\Zcal\ \big|\ \Phi(\orbit)=0\right\}\subset \bulletEOrbits=\left\{\orbit\in\oZcal\ \big|\ \Phi(\orbit)=0\right\},
\end{align*}
respectively. If $\gamma_x$ is the constant map with value $x\in M$ and $Y(x)=0$, then $(s,Y,\gamma_x)\in\bulletEOrbits$ for every $s\in\R^+$. The unit circle $S^1$ acts on $\bulletEOrbits$: for $\theta_0\in\R$ and $\orbit=(s,Y,\gamma)\in\bulletEOrbits$ we define $\orbit_{\theta_0}=(s,Y,\gamma_{\theta_0})$ by setting $\gamma_{\theta_0}(\theta)=\gamma(\theta+\theta_0)$. Since $\orbit_n=\orbit$ for $n\in\Z$, this gives an action of $\R/\Z$ on $\EOrbits$. The quotient of $\EOrbits$ by this action of $S^1$ is denoted by $\EOrbits/S^1$ and the class represented by $\orbit\in\EOrbits$ is denoted by $[\orbit]\in\EOrbits/S^1$. $\EOrbits/S^1$ is called the moduli space of (non-constant) periodic orbits and is fibered over $\Ycal$ via a map $p:\EOrbits/S^1\ra \Ycal$. At $\orbit\in\bulletEOrbits\subset\oZcal$ the tangent space to $\Ecal$ naturally decomposes as 
\begin{align*}
T_{\Phi(\orbit)}\Ecal=\R\oplus T_{Y_\orbit}\Ycal\oplus T_{\gamma_\orbit}\Xcal\oplus \Ecal_{\orbit}.
\end{align*}
Note that $T_Y\Ycal=\Ycal$ and $T_\gamma\Xcal= A^2(\gamma) $. Project the differential $d\Phi$ of $\Phi$ over $\Ecal_{\orbit}$ to obtain 
\begin{align*}
F_\orbit:\R\oplus \Ycal\oplus A^2(\gamma)\ra \Ecal_{\orbit}= A^1(\gamma).
\end{align*}
Let $\EOrbits^\circ$ denote the open subset of $\bulletEOrbits$ consisting of $\orbit\in\bulletEOrbits$ such that $F_\orbit$ is surjective. The intersection of $\Phi$ with the zero section of $\Ecal$ is transverse at every $\orbit\in\EOrbits^\circ$. In particular, $\EOrbits^\circ$ is cut out as a $C^{1}$ separable Banach manifold. The projection map $p_\circ=p|_{\EOrbits^\circ}:\EOrbits^\circ\ra \Ycal$ is  Fredholm  of index $1$.\\ 
 
The Sobolev spaces $\Fcal_\orbit=\R\oplus A^2(\gamma_\orbit)$ form a bundle $\pi_\Fcal:\Fcal\ra \oZcal$.  Denote the Fredholm restriction of $F_\orbit$ to $\Fcal_\orbit$ by $\Psi_\orbit$, for $\orbit\in\bulletEOrbits$. These linear maps give a bundle homomorphism $\Psi:\Fcal\ra \Ecal$ over $\bulletEOrbits$. $\orbit\in\EOrbits^\circ$ is regular for  $p_\circ$ if and only if $\Psi_\orbit$ is surjective (with $1$-dimensional kernel). For $\orbit\in\bulletEOrbits$, every tangent vector $\xi\in A^2(\gamma_\orbit) =T_{\gamma_\orbit}\Xcal$ gives a map from $S^1$ to $TM$ such that $\rho\circ\xi:S^1\ra M$ is the map $\gamma_\orbit$. Since $Y_\orbit$ is a section of $TM$, its differential $dY_\orbit:TM\ra T(TM)$ may be composed with $\xi$ to give 
\[dY_\orbit\circ \xi:S^1\ra T(TM).\] 
Applying $d\xi:TS^1\ra T(TM)$  to  $\partial_\theta:S^1\ra TS^1$ gives $d\xi\circ\partial_\theta:S^1\ra T(TM)$. The image of 
\begin{align*}
d\xi\circ\partial_\theta-{s_\orbit}\cdot dY_\orbit\circ\xi:S^1\ra T(TM)
\end{align*}
under $d\rho$ is $\dgam_\orbit-s_\orbit\cdot Y_\orbit\circ\gamma_\orbit$, which is zero. Therefore, $d\xi\circ\partial_\theta-s_\orbit\cdot dY_\orbit\circ\xi$ is in fact a vector in $TM$. By a direct computation in local coordinates, one finds the following lemma.

\begin{lem}\label{lem:differential}
For every  $\orbit\in\bulletEOrbits$, $s\in\R$, $Y\in\Ycal$ and  $\xi\in A^2(\gamma_\orbit)$,  we have 
\begin{align*}
F_\orbit(s,Y,\xi)= d\xi\circ\partial_\theta-s\cdot Y_\orbit\circ\gamma_\orbit-s_\orbit\cdot Y\circ\gamma_\orbit-s_\orbit\cdot dY_\orbit\circ\xi.
\end{align*}
\end{lem}
\begin{lem}\label{lem:constant-maps}
Let $\orbit=(s,Y,\gamma)\in\bulletEOrbits$, where $\gamma=\gamma_x$ is the constant loop based at $x\in M$. Then $\Psi_\orbit$ (resp. $F_\orbit$) is surjective if and only if  $d_xY$ has no eigenvalues in $\frac{2\pi i}{s}\cdot \Z$ (resp. in $\frac{2\pi i}{s}\cdot \Z\setminus\{0\}$).
\end{lem}
\begin{proof} Note that $Y(x)=0$. Identify $T_xM$ with $\R^n$, and note that under this identification, every $\xi\in A^2(\gamma)$ is in fact a function in $A^2(\R^n)$. We then have 
\begin{align*}
F_\orbit(r,Z,\xi)=\dot{\xi}-s\cdot (d_xY)\cdot \xi-s\cdot Z(x)\in A^1(\R^n),\quad\forall\ (r,Z,\xi)\in\R\times\Ycal\times A^2(\R^n).
\end{align*}
The surjectivity of $\Psi_\orbit$ follows, if one shows that the equation $\Psi_\orbit(\xi)=\dot{\xi}-s\cdot(d_xY)\cdot \xi=0$ has no non-zero solutions $\xi\in A^2(\C^n)$. Choose a complex linear change of variables in $\C^n=T_xM\otimes_\R\C$ and assume that  $d_xY$ is in the standard Jordan form
\begin{align*}
d_xY=\colvec[1]{\Lambda_1&0&\cdots&0\\
	0&\Lambda_2&\cdots&0\\
	\vdots&\vdots&\ddots&\vdots\\
	0&0&\cdots&\Lambda_m},\quad\text{where}\quad
\Lambda_k=\colvec[1]{\lambda_k&1&0&\cdots&0\\
0&\lambda_k&1&\cdots&0\\
0&0&\lambda_k&\cdots&0\\
\vdots&\vdots&\vdots&\ddots&\vdots\\
0&0&0&\cdots&\lambda_k}\in M_{n_k\times n_k}(\C),
\end{align*}
and $\sum_{k=1}^mn_k=n$. To solve $\Psi_\orbit(\xi)=0$, one should solve $\dot{f}_k=s\Lambda_k f_k$ for $f_k:S^1\ra \C^{n_k}$, which has non-trivial solution if and only if the kernel of $L_\lambda$ is trivial for $\lambda=s\cdot\lambda_k$, where
\begin{align*}
L_\lambda: A^2(S^1,\C)\ra A^1(S^1,\C),\quad L_\lambda(f):=\dot{f}-\lambda\cdot f.
\end{align*}
 If $f\in\Ker(L_{\lambda})$, then $f(\theta)=ae^{\lambda\theta}$ for some $a\in\C$. For such $f$ to be periodic (of period $1$) one needs to have $\lambda\in 2\pi i\Z$. Therefore, $\Psi_\orbit$ has non-trivial kernel if and only if $d_xY$ has an eigenvalue in $\frac{2\pi i}{s}\cdot\Z$. This completes the proof of the first claim. On the other hand, if $s\cdot\lambda_k=2\pi id$  and $d\neq 0$, then 
\begin{align*}
L_{2\pi i d}\Big(\frac{-c}{2\pi i d}\Big)=c\quad\Rightarrow\quad\left\{L_{2\pi i d}(f)+c\ |\ f\in A^2(\C),c\in\C\right\}=\Image(L_{2\pi i d})\neq A^1(\C).
\end{align*}
Therefore, $F_\orbit$ is not surjective. To complete the proof, we only need to show  that every $g\in A^1(\C)$ is of the form $c+\dot{f}$ for a constant $c\in\C$ and a function $f\in A^2(\C)$, which is clear. 
\end{proof}

It is important to interpret Lemma~\ref{lem:constant-maps} correctly. If $\orbit=(s,Y,\gamma_x)\in\bulletEOrbits\setminus \EOrbits$ is a constant periodic orbit and $d_xY$ has no eigenvalues in $\frac{2\pi i}{s}$, the lemma implies that a (sufficiently small) neighborhood of $\orbit$ in $\bulletEOrbits$ only consists of constant orbits. In other words, $(s,Y,\gamma_x)$ is a limit point of $\EOrbits$ only if $s=\frac{2\pi i d}{\lambda}$ for an eigenvalue $\lambda$ of $d_xY$ and some $d\in\Z$.   
\subsection{The holonomy map}
 Associated with every $\orbit=(s,Y,\gamma)\in\EOrbits$ is a {\emph{degree}} $d=d_\orbit\in\Z^+$ so that $\gamma=\mu\circ \varphi_d$, where $\varphi_d:S^1\ra S^1$ is the standard map of degree $d$ which corresponds to multiplication by $d$ in $\R$, $\orbit'=(s/d,Y,\mu)\in\EOrbits$ and $\mu$ is an embedding. We then write $\orbit=d\star \orbit'$.  We denote the subspace of $\EOrbits$ which consists of periodic orbits of degree $d\in\Z^+$ by $\EOrbits^d$, i.e. $\EOrbits=\coprod_{d=1}^\infty\EOrbits^d$. If $\orbit=d\star\orbit'\in\EOrbits^d$ is as above, and $N_{\mu}$ is the normal bundle of $\mu$ in $M$, $F_{\orbit}:\R\oplus\Ycal\oplus A^2(\gamma)\ra A^1(\gamma)$ may be composed with the projection from $\gamma^*TM$ to $N_\mu$ to give a map to $A^1(\varphi_d^*N_{\mu})$. Since $F_\orbit$ preserves the tangent direction, this gives the induced operators
\[F_\orbit^\perp:\Ycal\oplus A^2(\varphi_d^*N_\mu)\ra A^1(\varphi_d^*N_\mu)\quad\text{and}\quad\Psi_{\orbit}^\perp=\Psi_{d\star\orbit'}^\perp: A^2(\varphi_d^*N_{\mu})\ra A^1(\varphi_d^*N_{\mu}).\] 
\begin{lem}\label{lem:normal-operator}
The operator $\Psi_\orbit$ (resp. $F_\orbit$) is surjective if and only if $\Psi^\perp_\orbit$ (resp. $F_\orbit^\perp$) is surjective.
\end{lem}
\begin{proof}
Let $N_\mu=\mu^*TM/T_\mu$, where $T_\mu$ is the tangent bundle of $\mu$ in $M$. Every section $\xi$ of $\varphi_d^*T_\mu$ is of the form $\lambda\cdot\dot{\gamma}$, where $\lambda$ is a real-valued function defined on $S^1$. Applying Lemma~\ref{lem:differential} we obtain
\begin{align*}
\Psi_\orbit(r,\xi)&=d\xi\circ\partial_\theta-s \cdot dY\circ\xi-r\cdot Y\circ \gamma
=\Big(\dot{\lambda}-\frac{r}{s}\Big)\cdot\dot{\gamma}.
\end{align*}
Since every function in $A^1(\R)$ is of the form $\dot{\lambda}-(r/s)$ for some $r\in\R$ and some $\lambda\in A^2(\R)$, the map
\[\Psi_\orbit|_{T_\mu}:\R\oplus A^2(\varphi^*T_\mu)\ra A^1(\varphi^*T_\mu)\]
is surjective. Thus $\Psi_\orbit$ is surjective if and only if $\Psi^\perp_\orbit$ is surjective. The exact same proof works for $F_\orbit$. 
\end{proof}

Fix $\orbit=d\star\orbit'$  as above and $x=\gamma(0)$ on the image of $\mu$. Let $N_\orbit=N_{\orbit'}$ denote the fiber of the normal bundle $N_\mu$ at $x$. The holonomy map is the germ of the return map  $\fmap_{\orbit}:N_{\orbit}\ra N_{\orbit}$ of the flow of $Y$. The  differential of $\fmap_{\orbit}$ is a well-defined linear map $\Fmap_{\orbit}:N_{\orbit}\ra N_{\orbit}$, which is equal to $\Fmap_{\orbit'}^d=\Fmap_{\orbit'}\circ\cdots\circ\Fmap_{\orbit'}$. 
\begin{lem}\label{lem:regular-orbits}
If $\orbit'\in\EOrbits^1$  and $\orbit=d\star\orbit'$, the operator $\Psi_\orbit$ is surjective if and only if $1$ is not an eigenvalue of $\Fmap_\orbit$. $F_\orbit$ is surjective if and only if no eigenvalue $\lambda\neq 1$ of $\Fmap_{\orbit'}$ is a $d$-th root of unity. 
\end{lem}
\begin{proof}
Define the mapping cylinder of $\fmap_{\orbit}$ by 
\[C(\fmap_{\orbit})=\frac{\R\times N_{\orbit}}{\sim_{\fmap_{\orbit}}},\quad (\theta,\fmap_{\orbit}(x))\sim_{\fmap_{\orbit}} (\theta+1, x),\ \ \forall\ x\in N_{\orbit}\ \text{and}\ \theta\in\R.\]
Denote a neighborhood of the image of $\R\times\{0\}$ in $C(\fmap_{\orbit})$ by $B_{\orbit}$. We may choose the neighborhood $B_{\orbit}$ so that there is a $d$-fold covering map from $B_{\orbit}$ to a neighborhood of $\mu$ in $M$ and the pull-back of the vector field $Y$ under this covering map is identified with the vector field $\partial_\theta/s$, where $\theta$ parametrizes $\R$. Then the pull-back of the normal bundle $N_\mu$ may be identified with the mapping cylinder $C(\Fmap_\orbit)$ of the linearization $\Fmap_\orbit$ of $\fmap_\orbit$. A section $\xi$ of $N_\gamma$ is a function $\xi:\R\ra N_\orbit$ which satisfies $\xi(\theta+1)=\Fmap_\orbit(\xi(\theta))$. The operator $\Psi^\perp_\orbit$ is then given by $\Psi^\perp_\orbit(\xi)=\partial\xi/\partial\theta$ (By Lemma~\ref{lem:differential}). Thus, $\xi$ is in  $\Ker(\Psi^\perp_\orbit)$ if and only if it is constant, which can happen only if it is an eigenvector for $\Fmap_\orbit=\Fmap_{\orbit'}^d$ which corresponds to the eigenvalue $1$. This completes the proof of the first claim.\\

In order to prove the second claim, consider the Jordan form of $\Fmap_{\orbit'}$, as in the proof of Lemma~\ref{lem:constant-maps}. The blocks corresponding to the eigenvalues $\lambda$ which are not $d$-th roots of unity do not harm the surjectivity. $F_{\orbit}$ is surjective if and only if
\[\widetilde{L}_\lambda:A^2(\C)\oplus A^1_\lambda(\C)\ra A^1(\C),\quad\quad\widetilde{L}_\lambda(f,h):=\dot{f}-h,\] 
is surjective for every eigenvalue $\lambda$ of $F_{\orbit'}$ with $\lambda^d=1$, where $A^1_\lambda(\C)$ consists of $h\in A^1(\C)$ with $h(\theta+1/d)=\lambda\cdot h(\theta)$. If $\lambda\neq 1$, then the integral of $\widetilde{L}_\lambda(f,h)$  over $S^1$ is zero and $\widetilde{L}_\lambda$ is thus not surjective. On the other hand, constant functions are included in $A^1_1(\C)$ and $\widetilde{L}_1$ is surjective. Therefore, $F_{\orbit}$ is surjective if and only if no eigenvalue $\lambda\neq 1$ of $\Fmap_{\orbit'}$ satisfies $\lambda^d=1$. This completes the proof. 
\end{proof}

\begin{defn}\label{defn:rigidity}
 For $d\in\Z^+$, $\orbit\in\EOrbits$ is called  $d${\emph{-rigid}} if no $d$-th root of unity is an eigenvalue for $\Fmap_{\orbit}$.  We call $\orbit$  {\emph{super-rigid}} if it is $d$-rigid for every $d\in\Z^+$. A class $[\orbit]\in\EOrbits^1/S^1$ is called {\emph{super-rigid}} if $\orbit\in\EOrbits^1$ is super-rigid.  For a  $d$-rigid orbit $\orbit\in\EOrbits^1$,   let $\epsilon_d(\orbit)=\epsilon(\Fmap_\orbit^d)$ denote the sign  of $\det(\Fmap_{\orbit}^d-Id)$. 
\end{defn}

Consider a $(n-1)\times (n-1)$ real matrix $\Fmap$ such that $\Fmap^d-Id$ is non-singular for all $d\in\Z^+$. Let 
\[\lambda_1\leq\lambda_2\leq \cdots \leq \lambda_m,\nu_1,\overline{\nu}_1,\nu_2,\overline{\nu}_2,\ldots, \nu_k,\overline{\nu}_k\]
denote the set of eigenvalues of $\Fmap$, considering the multiplicities, where $m+2k=n-1$, $\lambda_i\in\R$ and $\nu_j\in\C\setminus \R$, for $i=1,\ldots,m$ and $j=1,\ldots,k$.  The eigenvalues of $\Fmap^d-Id$ are 
\[\lambda_1^d-1,\lambda_2^d-1, \ldots , \lambda_m^d-1, \nu_1^d-1,\overline{\nu}_1^d-1,\nu_2^d-1,\overline{\nu}_2^d-1,\ldots, \nu_k^d-1,\overline{\nu}_k^d-1.\]
Since $(\nu_j^d-1)(\overline{\nu}_j^d-1)>0$, it follows that 
\[\epsilon(\Fmap^d)=\sgn\Big(\prod_{j=1}^m(\lambda_j^d-1)\Big).\] Denote the number of $\lambda_j$s in $(-\infty,1)$, $(-1,1)$ and $(1,\infty)$ by $m_1$, $m_2$ and $m_3$, respectively (thus $m=m_1+m_3$). If $\lambda\in (-1,1)$, then $\lambda^d-1<0$ for all $d$, while for $\lambda>1$, $\lambda^d-1$ is always positive. Finally, for $\lambda<-1$ we have $\sgn(\lambda^d-1)=\sgn( (-1)^d)$. It follows that  $\epsilon(\Fmap^d)$ is given by
\[\epsilon(\Fmap^d)=\mathrm{sgn}(\det(\Fmap^d-Id))=(-1)^{m_{\epsilon(d)}},\]
where $\epsilon(d)=1$ for odd $d$ and $\epsilon(d)=2$ for even $d$.
For a  super-rigid $\orbit\in\EOrbits^1$, $\epsilon_d(\orbit)$ is thus equal to $\epsilon_1(\orbit)$ if $d$ is odd, and is equal to $\epsilon_2(\orbit)$ if $d$ is even.
 
\begin{defn}\label{defn:sign}
For a super-rigid $\orbit\in\EOrbits^1$ and $d\in\Z^+$, the  {\emph{weight}} $n(d\star\orbit)$ of $d\star\orbit$  (or the weight $n(\corbit)$ of the class $\corbit=[d\star\orbit]\in\EOrbits/S^1$ represented by it, is defined by 
\begin{align*}
n[d\star\orbit]=n(d\star\orbit):=\begin{cases}\epsilon_1(\orbit)&\text{if}\ \ d=1\\ (\epsilon_2(\orbit)-\epsilon_1(\orbit))/2&\text{if}\ \ d=2\\  0&\text{if}\ \  d>2 \end{cases}.
\end{align*}
  If  $\Gamma\subset\COrbits_Y\amalg(\EOrbits_Y/S^1)$ is the union of finitely many super-rigid classes, define $n(\Gamma):=\sum_{\corbit\in\Gamma}n(\corbit)$.
\end{defn} 
 
Define the set $\Ycal^{\supernice}$ of {\emph{super-nice}} vector fields (of class $C^1$) as the subset of $\Ycal^\infty\subset \Ycal$ which consists of the vector fields $Y$ with the property that all embedded closed $Y$-orbits are super-rigid. In the upcoming sections we prove the following theorem.

\begin{thm}\label{thm:super-nice-vector-fields}
$\Ycal^{\supernice}$ is a Bair subset of $\Ycal$. For  $Y\in\Ycal^{\supernice}$, $\COrbits_Y$ and $\EOrbits_Y/S^1$ are discrete, while each $\COrbits^d_Y$ is finite.
\end{thm}
 

\section{Periodic orbits, limit orbits and rigidity}
\subsection{Constant limits of periodic orbits}
Let us assume that $\{\orbit_k\}_{k=1}^\infty$ is a sequence in $\EOrbits$ with $\orbit_k=(s_k,Y_k,\gamma_k)$, such that $\gamma_k$ converges to the constant orbit $\gamma_x$ (with value $x\in M$), while $Y_k$ converges to $Y\in\Ycal$ and $s_k$ converges to $s\in\R^{\geq 0}$ as $k$ goes to infinity, with $Y(x)=0$. Fix a metric on $M$ and let $\|\dot\gamma_k\|_\infty$ denote the infinity norm of the derivative $\dot\gamma_k$. By passing to a subsequence, we may assume that $\dot\gamma_k/\|\dot\gamma_k\|_\infty$ converges to a map $\mu:S^1\ra T_xM$ with $\|\mu\|_\infty=1$. Moreover, it follows from the equations
\[\frac{\ddot{\gamma}_k}{\|\dot\gamma_k\|_\infty}=s_n\cdot d_{\gamma_k}Y_n\cdot\Big(\frac{\dot\gamma_k}{\|\dot\gamma_k\|_\infty}\Big),\quad\quad k=1,2,\cdots,\]
that $\mu$ satisfies $\dot\mu=s\cdot d_xY\cdot \mu$. Note that the integral of $\mu$ over $S^1$ is zero, and it is thus not constant. Therefore, $s\in\R^+$, and the image of $\mu$ is included in a plane $P\subset T_xM$, which remains invariant under $d_xY:T_xM\ra T_xM$. Moreover, the eigenvalues of $d_xY|_P:P\ra P$ are $\pm 2\pi i d/s$ for a positive integer $d\in\Z^+$. Therefore, $\corbit=(s,Y,x,P)$ belongs to $\Bcal$. Note that the limit point remains unchanged as the sequence $\{\orbit_k\}_{k=1}^\infty$ changes by the rotation action of $S^1$.\\

Based on the above discussion, $\EOrbits/S^1$ may be completed by taking its union with $\Bcal$. We will see in the next subsection that all points in $\Bcal$ are in fact limit points of sequences in $\EOrbits/S^1$. We abuse the notation and write $\oEOrbits/S^1$ for this completion, which is (as a set) a disjoint union
\[\oEOrbits/S^1=\Bcal\amalg(\EOrbits/S^1)=\coprod_{d\in\Z^+}(\oEOrbits^d/S^1).\]

We also define the {\emph{moduli space of ghost and periodic orbits}}
as the union 
\[\OModuli=\oCOrbits\cup_\Bcal(\oEOrbits/S^1)=\coprod_{d\in\Z^+}\OModuli^d=\coprod_{d\in\Z^+}\left(\oCOrbits^d\cup_{\Bcal^d}(\oEOrbits^d/S^1)\right).\]
This moduli space is equipped with a projection map $\pi:\OModuli\ra \Ycal$. As before, the pre-image of $Y\in\Ycal$ under $\pi$ is denoted by $\OModuli_Y=\coprod_{d\in\Z^+}\OModuli^d_Y$.\\

A boundary critical point $\corbit=(s,Y,x,P)\in\Bcal$ is called  singular if $d_xY$ is singular. Singular boundary points form a subset $\Bcal_s$ of $\Bcal$. Then $\Bcal_*=\Bcal\setminus \Bcal_s$ is open and dense. We also set 
\begin{align*}
&\oEOrbits_*/S^1=(\oEOrbits/S^1)\setminus \Bcal_s=\coprod_{d\in\Z^+}(\oEOrbits^d_*/S^1),\quad\\
&\oCOrbits_*=\oCOrbits\setminus \Bcal_s=\coprod_{d\in\Z^+}\oCOrbits^d_*\quad\quad\text{and}\quad\quad\OModuli_*=\OModuli\setminus\Bcal_s=\coprod_{d\in\Z^+}\OModuli^d_*.
\end{align*}
 In the following subsection, we show that $\oEOrbits^1_*/S^1$ is a $C^1$ Banach manifold with boundary $\Bcal^1_*$, and   $\OModuli^1_*$ may thus be equipped with the structure of a $C^1$ Banach manifold.   
\subsection{Gluing the moduli spaces of ghost orbits and periodic orbitst} In this subsection, we prove the following theorem.
\begin{thm}\label{thm:gluing-moduli-spaces}
The subset $\OModuli^1_*$ of $\OModuli_*$ may be equipped with the structure of a $C^1$ Banach manifold which  includes $\Bcal^1_*$ as a sub-manifold of codimension $1$, so that 
\[\OModuli^1_*\setminus \Bcal^1_*=\COrbits^1_*\amalg (\EOrbits^1_*/S^1).\] 
\end{thm}
\begin{proof}
Given $\corbit=(s,Y,x,P)\in\Bcal$,   the projection map $p^c:\Ncal^c_*\ra \Ycal$ admits a local inverse from a neighborhood $U$ of $Y$ in $\Ycal$ to a neighborhood $U''$ of $(Y,x)$ in $\Ncal^c$, which is given by $Z\mapsto (Z,x_Z)$ for $Z\in U$. Choose a coordinate chart around $x$, which corresponds to an open set $V\subset\R^n$ so that $x$ corresponds to $0\in V$.  By shrinking $U$, we may further assume that all points $x_Z$ (for $Z\in U$) are included in $V$. Let $\Xcal_V$ denote the space of $C^2$ embedding of $S^1$ in $V$ and set 
\[\EOrbits(U,V)=\left\{(t,s,Z,\gamma)\in [0,1]\times\R^+\times U\times \Xcal_V\ \big|\ \dot\gamma=t\cdot s\cdot Z\circ\gamma+(1-t)\cdot s \cdot d_{x_Z}Z\cdot (\gamma-x_Z)\right\}.\] 
Since $\Xcal_V$ only consists of embeddings, our earlier arguments may be repeated to show that $\EOrbits(U,V)$ is a $C^1$ Banach manifold which fibers over $[0,1]$, while all values of this projection map are regular. Denote the pre-image of $t\in[0,1]$ in $\EOrbits(U,V)$ by $\EOrbits_{t}=\EOrbits_t(U,V)$. Therefore, $\EOrbits_{1}$ is identified as a subset of $\EOrbits$. The tangent space to $\EOrbits(U,V)$ at $(t,s,Z,\gamma)$ consists of the vectors
\[(t',s',X,\mu)\in\R\times\R\times \Ycal\times A^2(\R^n)\] for which the following equation is satisfied:
\begin{align*}
\dot\mu&=(t's+ts')\cdot Z\circ \gamma+ (s'-t's-ts')\cdot d_{x_Z}Z\cdot (\gamma-x_Z)\\
&\quad\quad\quad+ (ts)\cdot  (X\circ\gamma+ d_\gamma Z\cdot \mu) +(s-ts)\cdot (d_{x_Z}X\cdot(\gamma-x_Z)+d_{x_Z}Z\cdot\mu).
\end{align*}
In particular, if we set $X(x)=Y_{Z,F}(x):=F(x-x_Z)$ for a linear map $F$ on $\R^n$, for $(1,0,X,\mu)$ to be in $T_{(t,s,Z,\gamma)}\EOrbits(U,V)$, it suffices to have
\begin{align*}
\dot\mu&=s\cdot (Z\circ \gamma-d_{x_Z}Z\cdot(\gamma-x_Z))+ d_{x_Z} Z\cdot \mu+F(\gamma-x_Z)+ts\cdot (d_\gamma Z-d_{x_Z}Z)\cdot\mu.
\end{align*}
Define the operator $\Phi_s:A^2(\R^n)\ra A^1(\R^n)$ by $\Phi_s(\mu):=\dot\mu-s\cdot d_0Y\cdot\mu$. Then $(1,0,Y_{Z,F},\mu)$ is in $T_{(t,s,Z,\gamma)}\EOrbits(U,V)$  if and only if we have 
\begin{align*}
\Phi_s(\mu)&=s\cdot\big(d_{x_Z}Z-d_0Y+t\cdot(d_\gamma Z-d_{x_Z}Z)\big)\cdot\mu\\ &\quad\quad\quad+s\cdot\big(Z\circ\gamma-d_{x_Z}Z\cdot (\gamma-x_Z)+ F\cdot (\gamma-x_Z)\big)\\
&=:s\cdot \big(B_{t,Z,\gamma}\cdot\mu+A_{Z,\gamma}+F\cdot (\gamma-x_Z)\big).
\end{align*}
Note that the norms of $B_{t,Z,\gamma}$ and $A_{Z,\gamma}$ are bounded above via 
\[\|B_{t,Z,\gamma}\|_\infty\leq b\|Z-Y\|_\infty+\|\gamma-x_Z\|_\infty\quad\text{and}\quad\|A_{Z,\gamma}\|_\infty\leq a\|\gamma-x_Z\|_\infty,\] 
 where $a$ and $b$ are constant. By choosing the open sets $U$ and $V$ sufficiently small, we may thus assume that the linear operator $\Phi_s-s\cdot B_{t,Z,\gamma}$ is arbitrarily close to $\Phi_s$. Let us denote the kernel of $\Phi_s$ by $K_s\subset A^2(\R^n)$. Using the standard Euclidean metric on $\R^n$, the orthogonal complement $K_s^\perp\subset A^2(\R^n)$ of $K_s$ may be defined, and $\Phi_s-B_{t,Z,\gamma}$ gives a bijective map from $K_s^\perp$ to its image $I_{t,s,Z,\gamma}\subset A^1(\R^n)$, which is a closed subspace of finite codimension and a small perturbation of the image $I_s$ of $\Phi_s$. In particular, if $I_s^\perp$ denotes the orthogonal complement of $I_s$ (which may be identified as the kernel of the adjoint operator $\Phi_s^*:A^2(\R^n)\ra A^1(\R^n)$ associated with $\Phi_s$), we have 
\[A^1(\R^n)=I_{t,s,Z,\gamma}\oplus I_s^\perp,\]
provided that $U$ and $V$ are sufficiently small. The vector $sA_{Z,\gamma}\in A^1(\R^n)$ is then of the form 
\[\Phi_s(\mu)-s\cdot B_{t,Z,\gamma}\cdot\mu-sA_{Z,\gamma}^\perp,\] where the norm of $A_{Z,\gamma}^\perp\in I_s^\perp$ is bounded above by $a'\|\gamma-x_Z\|_\infty^2$ for a constant $a'$. Let $F$ be the linear map which sends $\gamma-x_Z$ to $A_{Z,\gamma}^\perp$ and is zero on the orthogonal complement of $\gamma-x_Z$. The norm of $F$ is then bounded above by $b'\|\gamma-x_Z\|_\infty$ for a constant $b'$. Moreover, we obtain a unique tangent vector 
\[(1,0,Y_{Z,F},\mu)\in T_{(t,s,Z,\gamma)}\EOrbits(U,V).\] 
These tangent vectors give a vector field on $\EOrbits(U,V)$ which may be integrated to give the flow lines connecting a neighborhood of 
\[\EOrbits(U,V)\cap \big(\{0\}\times\R^+\times\{Y\}\times\Xcal_V\big)\subset\EOrbits_0\] 
to a neighborhood of 
\[\EOrbits(U,V) \cap \big(\{1\}\times\R^+\times\{Y\}\times\Xcal_V\big)\subset\EOrbits_1.\] 
 Note that $s\in\R^+$ remains constant along the flow lines. The above discussion gives a $C^1$ diffeomorphism from the intersection of a neighborhood of $\R^+\times \{(Y,\gamma_x)\}\subset \bulletEOrbits$  with $\EOrbits$ to the intersection of a neighborhood of 
 \[\{0\}\times\R^+\times\{(Y,\gamma_0)\}\subset\{0\}\times \R^{\geq 0}\times U\times A^2(\R^n)\]
  with $\EOrbits_0$. Of course, this local $C^1$ diffeomorphism from an open subset of $\EOrbits_1$ to an open subset of $\EOrbits_0$ respects the rotation action of $S^1$ and gives a corresponding $C^1$ diffeomorphism from an open subset of $\EOrbits_1/S^1$ to an open subset of $\EOrbits_0/S^1$.\\

$\EOrbits_0$ is a $C^1$ Banach manifold. If $(0,s,Z,\gamma)\in\EOrbits_0$, the equation $\dot\gamma=s\cdot d_{x_Z}Z\cdot\gamma$ has a solution. The proof of Lemma~\ref{lem:constant-maps} then implies that $\lambda=\frac{2\pi i}{s}$ is an eigenvalue of $d_{x_Z}Z$ and that the image of $\gamma$ is included in a plane $P\subset \R^n$ which is preserved by $d_{x_Z}Z$ and corresponds to $\lambda$. Therefore, $\EOrbits_0/S^1$ may be identified as an open subset of $\R^+\times \Bcal$. Informally, the above discussion gives a {\emph{blow up}} construction, which replaces the quotient of $\bulletEOrbits\setminus\EOrbits$ by the rotation action of $S^1$ with the $C^1$ Banach manifold $\Bcal$, so that we obtain the structure of a $C^1$ Banach manifold with boundary on $\oEOrbits^1_*/S^1$, and thus the structure of a $C^1$ Banach manifold on
  \[\OModuli^1_*=(\oEOrbits^1_*/S^1)\cup_{\Bcal^1_*} \oCOrbits^1_*.\]
  This completes the proof of the theorem.
\end{proof}
\begin{rmk}\label{rmk:gluing-result}
The same line of argument may be employed to show that  $\oEOrbits_*^d/S^1$ has the structure of a $C^1$ Banach manifold with boundary equal to $\Bcal^d_*$. Therefore, each $\OModuli^d_*$ has the structure of a $C^1$ Banach manifold. The projection map $\pi:\OModuli\ra \Ycal$ gives the index zero Fredholm projections
$\pi_*^d:\OModuli^d_*\ra\Ycal$ for $d\in\Z^+$, which are $C^1$ away from $\Bcal^d_*$, but only continuous at $\Bcal^d_*$. Nevertheless, $b^d=\pi^d_*|_{\Bcal^d_*}$ is also a Fredholm map of index $-1$.
\qed
\end{rmk}

\subsection{Super-rigidity for embedded periodic orbits}
Let $\{\orbit_i\}_{i=1}^\infty$ be a sequence in $\EOrbits$, representing distinct points of $\EOrbits/S^1$, so that $\{[\orbit_i]\}_{i=1}^n$ converges to a point $\corbit\in\OModuli_Y\subset\OModuli$. If $Y\in\Ycal^\infty$, it follows that $\Bcal_Y$ is empty. Therefore, $\corbit=[\orbit]\in\EOrbits/S^1$. Lemma~\ref{lem:normal-operator} and Lemma~\ref{lem:regular-orbits} imply that for $Y\in\Ycal^{\supernice}$, every $[\orbit]\in\EOrbits_Y/S^1$ is isolated. This observation implies the following lemma. 

\begin{lem}\label{lem:no-limits-for-super-nice-vf}
If $Y\in\Ycal^{\supernice}$, then 
$\EOrbits_Y/S^1$ is discrete as a subspace of either of $\oZcal/S^1$ and $\OModuli$. 
\end{lem}

Given $\orbit\in\EOrbits^1$, the linearized holonomy map $\Fmap_\orbit:N_\orbit\ra N_\orbit$ is well-defined. Choosing an identification of $N_\orbit$ with $\R^{n-1}$ (as vector spaces), the conjugacy class of $\Fmap_\orbit:\R^{n-1}\ra\R^{n-1}$ is well-defined and only depends on the class $[\orbit]\in\EOrbits^1/S^1$. In particular, the characteristic polynomial 
\[\chi_{[\orbit]}=\chi_\orbit:=\det(\hslash\cdot I-\Fmap_\orbit)=\hslash^{n-1}+a^1_{\orbit}\hslash^{n-2}+\cdots+a^{n-2}_{\orbit}\hslash+a^{n-1}_{\orbit}\in\R[\hslash]\]
is well-defined. Alternatively, we may think of $\chi_{\orbit]}$ as the point $(a^1_{\orbit},\ldots,a^{n-1}_\orbit)\in\R^{n-1}$. This gives a $C^1$ map $\chi:\EOrbits^1/S^1\ra \R^{n-1}$ (defined by sending $[\orbit]\in\EOrbits^1/S^1$ to $\chi_{[\orbit]}\in\R^{n-1}$) which is called the {\emph{characteristic function}}.

\begin{lem}\label{lem:characteristic-function}
Every value in $\R^{n-1}$ is regular  for the characteristic function $\chi:\EOrbits^1/S^1\ra \R^{n-1}$. In particular, the pre-image of every codimension $k$ submanifold $\Gamma$ of $\R^{n-1}$ under $\chi$ is a codimension $k$ Banach submanifold  $\EOrbits^\Gamma/S^1=\chi^{-1}(\Gamma)\subset\EOrbits^1/S^1$. It is equipped with a Fredholm projection map of index $-k$ to $\Ycal$, and the set of its regular values is a Bair subset $\Ycal^\Gamma\subset\Ycal$.  
\end{lem} 
\begin{proof}
For $\orbit=(s,Y,\gamma)\in\EOrbits^1$, the tangent space $T_\orbit\EOrbits^1$ is a subspace of 
\[T_s\R^+\oplus T_Y\Ycal\oplus T_\gamma\Xcal=\R\oplus\Ycal\oplus A^2(\gamma),\]
and includes the vector space generated by the triples $(0,Z,0)$ such that the restriction of $Z\in\Ycal$ to the image of $\gamma$ is zero. Under the identification of a neighborhood of $\gamma$ with the open subset $B_{\orbit}$ in the mapping cone $C(\fmap_{\orbit})$, the restriction of $Z$ to $B_{\Fmap_{\orbit}}$ is given as a map $Z:\R\times N_{\orbit}\ra \R\times N_{\orbit}$, which is linear in the second variable and satisfies
\begin{equation}\label{eq:vector-field-period}
\Fmap_{\orbit}\left(Z(\theta+1,x)\right)=Z\left(\theta,\fmap_{\orbit}(x)\right)\quad\quad\forall\ \ \theta\in\R.
\end{equation}
Choose a bump function $\rho:[0,1]\ra \R^{\geq 0}$, supported in a neighborhood of $1/2$, so that $\int_0^1\rho(\theta)d\theta=1$. Fix an arbitrary linear map $\gfrak:N_\orbit\ra N_\orbit$ and define the restriction of $Z=Z_\gfrak$ to $[0,1]\times N_\orbit$ by 
\[Z(\theta,x):=\left(0,\rho(\theta)\cdot \gfrak(x)\right)\in\R\times N_\orbit,\quad\quad\forall\ \ \theta\in[0,1]\ \ \text{and}\ \ x\in N_\orbit.\]
Note that for $\theta$ in small neighborhoods of $0$ and $1$, $Z(\theta,x)$ is identically zero. We may thus extend $Z$  to $\R\times N_\orbit$ using Equation~\ref{eq:vector-field-period}. Correspondingly, we obtain a path of vector fields $\vY=\{\vY_t\}_{t\in(-\epsilon,\epsilon)}$ which is defined in $B_\orbit$ by  $\vY_t(\theta,x)=t Z(\theta,x)+\partial_\theta/s$. The flow lines of $\vY_t$ are given by 
\[\gamma_t^x(\theta)=\left(\theta,e^{t(\int_0^\theta\rho(r)dr)\cdot\gfrak}\cdot x\right)\quad\quad\forall\ \ \theta\in \R\ \ \text{and}\ \ x\in N_\orbit.\]
The image of $x\in N_\orbit=\{0\}\times N_\orbit$ under the flow is given by 
\[\gamma_t^x(1)=\left(1,e^{t\gfrak}\cdot x\right)=\left(0,\fmap_\orbit^{-1}\left(e^{t\gfrak}\cdot x\right)\right).\]
The holonomy map associated with the vector field $\vY_t$ is given by $e^{-t\gfrak}\circ \fmap_\orbit$ and its linearization is equal to $e^{-t\gfrak}\circ \Fmap_\orbit$. Differentiating this map with respect to $t$ and evaluating at $t=0$ gives 
\[\frac{\partial}{\partial t}\left(e^{-t\gfrak}\circ \Fmap_\orbit\right)|_{t=0}(x)=-\left(\gfrak\cdot e^{-t\gfrak}\right)|_{t=0}\cdot \Fmap_\orbit(x)=-\gfrak\cdot\Fmap_\orbit(x).\]
Note that $\Fmap=\Fmap_\orbit$ is viewed as an invertible matrix in $M_{(n-1)\times (n-1)}(\R)$. Therefore, the image of 
\[d_{[\orbit]}\chi:T_{[\orbit]}\EOrbits^1/S^1\ra T_{\chi(\orbit)}\R^{n-1}=\R^{n-1}\]
includes all points of the form $b_\gfrak=(b^1_\gfrak,\ldots,b^{n-1}_\gfrak)$, where for $\gfrak\in M_{(n-1)\times(n-1)}(\R)$ we set
\[b^1_\gfrak\hslash^{n-2}+\cdots+b^{n-2}_\gfrak\hslash+b^{n-1}_\gfrak=\lim_{t\ra 0}\frac{\det\left(\hslash\cdot I-(I-t\gfrak)\Fmap\right)-\det\left(\hslash\cdot I-\Fmap\right)}{t}.\] 
It is then an easy exercise in linear algebra to conclude (from the invertibility of $\Fmap$) that by varying $\gfrak$ in $M_{(n-1)\times(n-1)}$, $b_\gfrak=d_{[\orbit]}\chi(0,Z_\gfrak,0)$ covers all of $\R^{n-1}$. This completes the proof of the lemma. 
\end{proof}

For every $\lambda\in \C$ and $n\geq 2$, set
\[\Gamma_\lambda:=\{(a^1,\ldots,a^{n-1})\in\R^{n-1}\ \big|\ 
a^1\lambda^{n-2}+a^2\lambda^{n-3}+\cdots+a^{n-1}=0 \}.\]
We also set $\Ycal^\lambda=\Ycal^{\Gamma_\lambda}$. For $\lambda\in\R$, $\Gamma_\lambda$ is a linear subspace of $\R^{n-1}$ of codimension $1$, while for $\lambda\in\C\setminus\R$, $\Gamma_\lambda$ is a linear subspace of $\R^{n-1}$ of codimension $2$. 

\begin{proof}[Proof of Theorem~\ref{thm:super-nice-vector-fields}]
Let $\Lambda$ denote the union of $\{\infty\}$ with all roots of unity in $\C$. Every vector field $Y$ in the Bair subset $\bigcap_{\lambda\in\Lambda} \Ycal^\lambda$ of $\Ycal$ belongs to $\Ycal^{\supernice}$. In fact, the intersection is a subset of $\Ycal^\infty$ and  if  $\orbit\in \EOrbits^1_Y$ is not super-rigid, then $1$ is an eigenvalue of $\Fmap_{\orbit}^d$ for some $d\in\Z^+$. This means that for a divisor $d'$ of $d$ and a primitive $d'$-th root of unity $\lambda$, $\lambda$ is an eigenvalue of $\Fmap_\orbit$. Therefore, $[\orbit]\in\EOrbits^\lambda_Y/S^1$, which is empty by  Lemma~\ref{lem:characteristic-function}	. This contradiction proves the claim. For $Y\in\bigcap_{\lambda\in\Lambda}\Ycal^\lambda\subset \Ycal^{\supernice}$, the fiber $\Bcal_Y$, consisting of boundary $Y$-orbits is empty and 
\[\OModuli_Y=\pi^{-1}(Y)=\COrbits_Y\amalg (\EOrbits_Y/S^1)\]
 is discrete. If we fix the degree, it is clear that there are only finitely many ghost $Y$-orbits. This completes the proof.
\end{proof} 
 

\section{Paths of vector fields and cobordisms}
\subsection{Cobordisms between moduli spaces of periodic orbits}
Given $Y_0,Y_1\in\Ycal$, let $\Ycal_{Y_0,Y_1}$ denote the space of $C^1$ sections $\vY$ of $\rho_M^*TM\ra [0,1]\times M$, which satisfy $\vY(0,\cdot)=Y_0$ and $\vY(1,\cdot)=Y_1$. Here $\rho_M:[0,1]\times M\ra M$ denotes the projection map over $M$. We sometimes denote $\vY(t,\cdot)\in\Ycal$ by $\vY_t$. In particular, $\vY_0=Y_0$ and $\vY_1=Y_1$.  The moduli space
\[\oCOrbits_{Y_0,Y_1}=\left\{(t,s,\vY,x,P)\in [0,1]\times\R^+\times \Ycal_{Y_0,Y_1}\times G\ \big|\ (s,\vY_t,x,P)\in\oCOrbits\right\}\] 
is equipped with an evaluation map $\efrak_{Y_0,Y_1}:\oCOrbits_{Y_0,Y_1}\ra \oCOrbits$, defined by $\efrak(t,s,\vY,x,P)=(s,\vY_t,x,P)$. The trace map $\tfrak_{Y_0,Y_1}:\oCOrbits_{Y_0,Y_1}\ra \R$  is defined as the composition  $\tfrak\circ\efrak_{Y_0,Y_1}$. We then set 
\[\Bcal_{Y_0,Y_1}:=\tfrak_{Y_0,Y_1}^{-1}(0)\quad\text{and}\quad \COrbits_{Y_0,Y_1}:=\tfrak_{Y_0,Y_1}^{-1}(-\infty,0)=\oCOrbits_{Y_0,Y_1}\setminus \Bcal_{Y_0,Y_1}.\]  
Let us denote the pull back of $\Ecal$ (from $\Xcal$) over 
\[\Zcal_{Y_0,Y_1}=[0,1]\times \R^+\times \Ycal_{Y_0,Y_1}\times \Xcal\] 
by $\Ecal_{Y_0,Y_1}\ra \Zcal_{Y_0,Y_1}$. The zero locus of the section 
\[\vPhi_{Y_0,Y_1}:\Zcal_{Y_0,Y_1}\ra \Ecal_{Y_0,Y_1},\quad \vPhi_{Y_0,Y_1}(t,s,\vY,\gamma):=\dgam-s\cdot \vY_t\circ \gamma,\] forms the space $\EOrbits_{Y_0,Y_1}$. Given $\vorbit=(t,s,\vY,\gamma)\in\EOrbits_{Y_0,Y_1}$, we set $t_\vorbit=t,s_\vorbit=s,\vY_\vorbit=\vY$ and $\gamma_\vorbit=\gamma$. For $\vY\in \Ycal_{Y_0,Y_1}$, by a periodic $\vY$-orbit we mean a periodic  $\vY_t$-orbit for some $t\in [0,1]$. At  $\vorbit\in\EOrbits_{Y_0,Y_1}$ we may compose $d\vPhi_{Y_0,Y_1}$ with projection over the fiber of $\Ecal_{Y_0,Y_1}$, and obtain an operator $\vF_\vorbit$. Set $\EOrbits^d_{Y_0,Y_1}=e_{Y_0,Y_1}^{-1}(\EOrbits^d)$ for every $d\in\Z^+$. The space $\EOrbits_{Y_0,Y_1}/S^1$ may be completed to $\oEOrbits_{Y_0,Y_1}/S^1$ by adding $\Bcal_{Y_0,Y_1}$. We may also set 
\[\OModuli_{Y_0,Y_1}=\oCOrbits_{Y_0,Y_1}\cup_{\Bcal_{Y_0,Y_1}}\oEOrbits_{Y_0,Y_1}=\coprod_{d\in\Z^+}\OModuli^d_{Y_0,Y_1}.\] 
The projection map from $\OModuli_{Y_0,Y_1}$ to $\Ycal_{Y_0,Y_1}$ is denoted by $\pi_{Y_0,Y_1}$, and its restriction to $\OModuli^d_{Y_0,Y_1}$ is denoted by $\pi^d_{Y_0,Y_1}$. We may also define an evaluation map by 
\[e_{Y_0,Y_1}:\OModuli_{Y_0,Y_1}\ra \OModuli,\quad\quad e_{Y_0,Y_1}(t,s,\vY,\cdot):=\left(s,\vY_t,\cdot\right)\quad\forall\ (t,s,\vY,\cdot)\in\OModuli_{Y_0,Y_1}.\] 
A point $\ocorbit=(t,s,\vY,x,P)\in\Bcal_{Y_0,Y_1}$ is called singular if $e_{Y_0,Y_1}(\ocorbit)\in\Bcal_s$. The set of singular points in $\Bcal_{Y_0,Y_1}$ is denoted by $\Bcal_{Y_0,Y_1,s}$. We may also set 
\begin{align*}
&\Bcal_{Y_0,Y_1,*}=\Bcal_{Y_0,Y_1}\setminus\Bcal_{Y_0,Y_1,s}=\coprod_{d\in\Z^+}\Bcal^d_{Y_0,Y_1,*}\quad\text{and}\\
&\OModuli_{Y_0,Y_1,*}=\OModuli_{Y_0,Y_1}\setminus\Bcal_{Y_0,Y_1,s}=\coprod_{d\in\Z^+}\OModuli^d_{Y_0,Y_1,*}.
\end{align*}
The following theorem follows from the transversality of the intersection of $\vPhi_{Y_0,Y_1}$ with the zero section of $\Ecal_{Y_0,Y_1}$, c.f.  \cite[Theorem~5.1]{Ef-p} and the arguments of Lemma~\ref{lem:blow-up} and Theorem~\ref{thm:gluing-moduli-spaces}.

\begin{thm}\label{thm:cobordism-transversality} 
For $Y_0,Y_1\in\Ycal^{\supernice}$, the metric space $\OModuli^d_{Y_0,Y_1,*}$ has the structure of a $C^1$ Banach manifold and includes $\Bcal^d_{Y_0,Y_1,*}$ as a codimension $1$ submanifold. The restriction of the projection map $\pi_{Y_0,Y_1}^d$ to $\OModuli^d_{Y_0,Y_1,*}\setminus \Bcal^d_{Y_0,Y_1,*}$ is a Fredholm map of index $1$,  for all $d\in\Z^+$. Moreover, the projection map $b_{Y_0,Y_1}:\Bcal_{Y_0,Y_1}\ra \Ycal_{Y_0,Y_1}$ is Fredholm of index $0$. 
\end{thm} 
Given $Y_0,Y_1\in\Ycal^\supernice$, there is a Bair subset of $\Ycal_{Y_0,Y_1}$ such that for every $\vY$ is this subset, $\Bcal_\vY\cap\Bcal_{Y_0,Y_1,s}$ is empty. Let $\Ycal^\infty_{Y_0,Y_1}$ denote the intersection of the aforementioned Bair subset of $\Ycal_{Y_0,Y_1}$ with the set of regular values of both 
\[\pi^1_{Y_0,Y_1,*}:\OModuli^1_{Y_0,Y_1,*}\ra \Ycal_{Y_0,Y_1}\quad\text{and}\quad b_{Y_0,Y_1}:\Bcal_{Y_0,Y_1}\ra \Ycal_{Y_0,Y_1}.\] 
 For $\Ycal^\infty_{Y_0,Y_1}$, $\OModuli^d_\vY$ is a $1$-manifold and $\Bcal^d_{\vY}$ is its $0$-sub-manifold, which is disjoint from $\Bcal_{Y_0,Y_1,s}$.  Furthermore, $\OModuli_{\vY}$ is complete.\\

We may define a corresponding characteristic function by
\[\chi_{Y_0,Y_1}:\EOrbits^1_{Y_0,Y_1}/S^1\ra\R^{n-1},\quad\quad \chi_{Y_0,Y_1}([\vorbit]):=\chi(e_{Y_0,Y_1}([\vorbit]))\quad\forall\ \ [\vorbit]\in\EOrbits^1_{Y_0,Y_1}/S^1.\]
The proof of Lemma~\ref{lem:characteristic-function} may be repeated to prove the following lemma.
\begin{lem}\label{lem:characteristic-function-path}
Given $Y_0,Y_1\in\Ycal^{\supernice}$, every value in $\R^{n-1}$ is  regular  for the characteristic function $\chi_{Y_0,Y_1}$. In particular, the pre-image of every codimension $k$ submanifold $\Gamma$ of $\R^{n-1}$ under $\chi_{Y_0,Y_1}$ is a codimension $k$ Banach submanifold  
\[\EOrbits^\Gamma_{Y_0,Y_1}/S^1=\chi_{Y_0,Y_1}^{-1}(\Gamma)\subset\EOrbits_{Y_0,Y_1}^1/S^1.\] 
This submanifold is equipped with a Fredholm projection map of index $1-k$ to $\Ycal_{Y_0,Y_1}$, and the set of its regular values is a Bair subset $\Ycal^\Gamma_{Y_0,Y_1}\subset\Ycal_{Y_0,Y_1}$.
\end{lem}

\subsection{Super-nice paths of vector fields}
As before, we set $\Ycal^\lambda_{Y_0,Y_1}=\Ycal^{\Gamma_\lambda}_{Y_0,Y_1}$.  Then, for 
\[\vY\in\bigcap_{\lambda\in\Lambda}\Ycal^\lambda_{Y_0,Y_1}\subset\Ycal_{Y_0,Y_1},\] 
the $\vY$-orbits in  the $1$-manifold  $\OModuli^1_\vY$ which are not super-rigid form a $0$-submanifold $\EOrbits^\Lambda_\vY/S^1\subset\EOrbits^1_\vY/S^1$.  For every point $[\vorbit]$ in this  $0$-submanifold, the only roots of unity which can be eigenvalues of $\Fmap_\vorbit$ are $\pm1$. We may further refine the Bair intersection $\bigcap_{\lambda\in\Lambda}\Ycal^\lambda_{Y_0,Y_1}$ to avoid a few unwanted situations. For $r\in\R$, we define the linear functions $L_r,L'_r:\R^{n-1}\ra\R$ at $a=(a^1,\ldots,a^{n-1})\in\R^{n-1}$ by 
\[L_r(a):=r^{n-1}+\sum_{j=1}^{n-1}r^{n-j-1}a^j,\quad \text{and}\quad L'_r(a):=(n-1)r^{n-2}+\sum_{j=1}^{n-1}(n-j-1)r^{n-j-2}a^j.\] 
For every pair of distinct real numbers $r,r'\in\R$, consider the codimension $2$ linear subspaces 
\begin{align*}
\Gamma_{r,r}=\big\{a\in\R^{n-1}\ \big|\ L_r(a)=L'_{r}(a)=0\big\}\quad\text{and}\quad \Gamma_{r,r'}=\big\{a\in\R^{n-1}\ \big|\ L_{r}(a)=L_{r'}(a)=0\big\}
\end{align*}
of $\R^{n-1}$. For every $\vY$ in the Bair subset \[\Ycal^{\imath}=\Ycal^{\Gamma_{1,1}}\cap\Ycal^{\Gamma_{1,-1}}\cap\Ycal^{\Gamma_{-1,-1}}\subset\Ycal,\] and every $\vorbit\in\EOrbits^1$, the total multiplicity of $1$ and $-1$ as eigenvalues of $\Fmap_\vorbit$ is either $0$ or $1$. Moreover, for every orbit $\vorbit$ in in the open subset
\[\Hcal_{Y_0,Y_1}=\EOrbits^{\Gamma_{-1}}\setminus \left(\EOrbits^{\Gamma_{1,-1}}\cup \EOrbits^{\Gamma_{-1,-1}}\right)\subset \EOrbits^{\Gamma_{-1}}\]
 the kernel of $\Fmap_\vorbit^2-Id$ is $1$-dimensional. Since $1$ is not an eigenvalue of $\Fmap_\vorbit$, the projection  map
\[p_\vY:\EOrbits^1_\vY/S^1\ra [0,1]\] 
is regular at $[\vorbit]$  and admits a local inverse 
\[\tau_1:(t-\epsilon,t+\epsilon)\ra \EOrbits^1_\vY/S^1\]
with the property that $p_\vY\circ\tau_1$ is the identity map on the interval $(t-\epsilon,t+\epsilon)$ and $\tau_1(t)=\vorbit$.
For $\epsilon$ sufficiently small and $r\in(t-\epsilon,t+\epsilon)$, the vector field $\vY_r$ determines a return map $\fmapp_r:N_\vorbit\ra N_\vorbit$ with $\fmapp_t=\fmap_\vorbit$, and $\tau_1(r)$ gives the unique fixed point $z(r)$ of $\fmapp_r$ in a neighborhood of $0\in N_\vorbit$. Subsequently, we may define a map
\[\vfmapp_\vorbit:(t-\epsilon,t+\epsilon)\times N_\vorbit\ra N_\vorbit,\quad\quad \vfmapp_\vorbit(r,z):=\fmapp_r(z+z(r))-z(r).\]
Note that $\vfmapp_\vorbit(r,0)=0$, while $0$ is the unique fixed point of $\vfmapp_\vorbit(r,\cdot)$ is a neighborhood of the origin in $N_\vorbit$. Then, every point in a neighborhood of $2\star\vorbit$ in $\OModuli_\vY$ corresponds to a pair $(r,z)\in (t-\epsilon,t+\epsilon)\times N_\vorbit$ such that $\vfmapp_\vorbit(r,\vfmapp_\vorbit(r,z))=z$.\\

 Note that $1$ is not an eigenvalue of $\Fmap_\vorbit$ while $-1$ has multiplicity $1$ as an eigenvalue of $\Fmap_\vorbit$. Thus, in an appropriate identification of $N_\vorbit$ with $\R^{n-1}$, $\Fmap_\vorbit$ is given by a matrix 
\begin{equation}\label{eq:-holonomy-map}
\Fmap_\vorbit:\R\times\R^{n-2}\ra\R\times\R^{n-2},\quad\Fmap_\vorbit=\left(\begin{array}{c|c}-1&0\\ \hline 0&Q\end{array}\right),
\end{equation}
where $\pm 1$ are not eigenvalues for $Q$. Under this identification of $N_\vorbit$ with $\R^{n-1}$, every point $z\in N_\vorbit$ is given by coordinates $(x,y)$, with $x\in\R$ and $y\in\R^{n-2}$ and the holonomy map $\vfmapp_\vorbit$ is identified as 
\[\vfmapp_\vorbit=(\vgmapp,\vhmapp):(t-\epsilon,t+\epsilon)\times\R\times\R^{n-2}\ra \R\times \R^{n-2}.\]
Then $\vfmapp_\vorbit(r,\vfmapp_\vorbit(r,z))=z$ for  $r\in (t-\epsilon, t+\epsilon)$ and $z=(x,y)\in\R\times\R^{n-1}$, if and only if
\[\vgmapp(r,\vgmapp(r,x,y),\vhmapp(r,x,y))-x=0\quad\text{and}\quad h(r,x,y):=\vhmapp(r,\vgmapp(r,x,y),\vhmapp(r,x,y))-y=0.\]
Note that 
\[h:(t-\epsilon,t+\epsilon)\times\R\times\R^{n-2}\ra\R^{n-2}\quad\text{and}\quad d_{(t,0,0)}h=\left(\begin{array}{ccc}\star&0&Q^2-Id\end{array}\right).\]
Thus, $d_{(t,0,0)}h$ is a full-rank matrix and the equation $h(r,x,y)=0$ may be uniquely solved for $y$ in a neighborhood of $(t,0,0)$. This gives a function $y=y(r,x)$ so that $h(r,x,y(r,x))=0$. Define 
\[\gmap_\vorbit(r,x)=\vgmapp(r,x,y(r,x)),\quad\text{and}\quad\hmap_\vorbit(r,x)=\vhmapp(r,x,y(r,x)).\]
Let us assume that $(r,z)=(r,x,y(r,x))$ is such that $\vfmapp_\vorbit(r,\vfmapp_\vorbit(r,z))=z$. Then the same is true for $(r,z')=(r,\vfmapp_\vorbit(r,z))$ and in particular, we have $\hmap_\vorbit(r,x)=y(r,\gmap_\vorbit(r,x))$. Therefore, we find
\begin{align*}
x&=\vgmapp(r,\vgmapp(r,x,y(r,x)),\vhmapp(r,x,y(r,x)))\\ &=\vgmapp(r,\gmap_\vorbit(r,x),y(r,\gmap_\vorbit(r,x)))=\gmap_\vorbit(r,\gmap_\vorbit(r,x)).
\end{align*}
Note that $\gmap_\vorbit$ is in fact defined as a map
\[\gmap_\vorbit:(t-\epsilon,t+\epsilon)\times \Ker(\Fmap_\vorbit+Id)\ra \Coker(\Fmap_\vorbit+Id)\quad\text{with}\quad\gmap_\vorbit(r,0)=0,\]
and the points in a neighborhood of $2\star\vorbit$ in $\OModuli_\vY$ are in correspondence with the solutions to $\gmap_\vorbit(r,\gmap_\vorbit(r,x))=x$. Moreover, note that $\partial_x\gmap_\vorbit(t,0)=-1$. Let us define 
\[p_\vorbit:(t-\epsilon,t+\epsilon)\times\Ker(\Fmap_\vorbit+Id) \ra \Coker(\Fmap_\vorbit+ Id)\] 
by setting $p_\vorbit(r,x)=\gmap_\vorbit(r,\gmap_\vorbit(r,x))-x$. The above discussion shows that the zero locus of $p_\vorbit$ gives a Kuranishi model for a neighborhood of $2\star\vorbit$ in $\OModuli_\vY$. Since $p_\orbit(r,0)=0$, we may write $p_\vorbit(r,x)=xq_\vorbit(r,x)$. If we denote the differential $\frac{\partial^{i+j}\gmap_\vorbit}{\partial r^i\partial x^j}$ by $\gmap_{r^ix^j}$, one computes
\begin{align*}
q_\vorbit(t,0)=\frac{\partial p_\vorbit}{\partial x}\big|_{(t,0)}&=\left(\gmap_x(r,\gmap(r,x))\cdot\gmap_x(r,x)-1\right)\big|_{(t,0)}=1-1=0,\quad\text{and}\\
\frac{\partial q_\vorbit}{\partial x}(t,0)=\frac{1}{2}\cdot\frac{\partial^2 p_\vorbit}{\partial x^2}\big|_{(t,0)}&=\frac{1}{2}\cdot\left(\gmap_{x^2}(r,\gmap(r,x))\cdot \gmap_x(r,x)^2+\gmap_x(r,\gmap(r,x))\cdot\gmap_{x^2}(r,x)\right)\big|_{(0,0)}\\
&=\frac{1}{2}\cdot\left(\gmap_{x^2}(0,0)\cdot (-1)^2+(-1)\cdot\gmap_{x^2}(0,0)\right)=0.
\end{align*}
The structure of the Kuranishi model thus depends on the non-triviality of the (well-defined) vectors
\[d_r\Fmap_\vorbit:=\frac{\partial q_\vorbit}{\partial r}(t,0)\in\Coker(\Fmap_\vorbit+Id)=:\Ccal_{\vorbit}\quad\text{and}\quad d_{\xi}^2\Fmap_\vorbit:=\frac{\partial^2 q_\vorbit}{\partial x^2}(t,0)\in\Coker(\Fmap_\vorbit+Id)=\Ccal_{\vorbit}.\] 
The vector spaces $\Ker(\Fmap_\vorbit+Id)$ and $\Coker(\Fmap_\vorbit+Id)$ sit together to form two line bundles $\Kcal_{Y_0,Y_1}$ and  $\Ccal_{Y_0,Y_1}$ over $\Hcal_{Y_0,Y_1}\subset\EOrbits^1_{Y_0,Y_1}$.  The above discussion gives a section 
\[d_r\Fmap_{Y_0,Y_1}:\Hcal_{Y_0,Y_1}\ra \Ccal_{Y_0,Y_1}\]
of the line bundle $\Ccal_{Y_0,Y_1}\ra \Hcal_{Y_0,Y_1}$ and a homomorphism  
\[d_x^2\Fmap_{Y_0,Y_1}:\Kcal_{Y_0,Y_1}\ra \Ccal_{Y_0,Y_1}\]
of line bundles over $\Hcal_{Y_0,Y_1}$, which are (respectively) defined by 
\begin{align*}
&d_r\Fmap_{Y_0,Y_1}(\vorbit):=(\vorbit;d_r\Fmap_\vorbit)\in \Ccal_\vorbit\quad\text{and}\quad d_x^2\Fmap_{Y_0,Y_1}(\vorbit;\xi):=(\vorbit;d_{\xi}^2\Fmap_{\vorbit})\in\Ccal_{\vorbit}.
\end{align*}

\begin{thm}\label{thm:super-nice-paths}
For  $Y_0,Y_1\in\Ycal^{\supernice}$, there is a Bair subset 
\[\Ycal^{\supernice}_{Y_0,Y_1}\subset \bigcap_{\lambda\in\Lambda}\Ycal^\lambda_{Y_0,Y_1}\subset\Ycal^*_{Y_0,Y_1}\] (of {\emph{super-nice}} paths) such that for every $\vY\in \Ycal^{\supernice}_{Y_0,Y_1}$ the following are satisfied:
\begin{itemize}
\item $\vY$ is a regular value of the projection  map $\pi_{Y_0,Y_1}:\OModuli^1_{Y_0,Y_1}\setminus \Bcal^1_{Y_0,Y_1}\ra \Ycal_{Y_0,Y_1}$ and $\OModuli^1_\vY$ has thus the structure of a $C^1$  $1$-manifold.
\item $\vY$ is a regular value of the projection $b_{Y_0,Y_1}$ and $\Bcal^1_\vY$ is a $0$-submanifold of $\OModuli^1_\vY$.
\item For every $\vorbit\in\EOrbits^1_\vY$, there are no roots $\lambda\neq \pm 1$ of unity  which are eigenvalues of the linearized holonomy map $\Fmap_\vorbit$.
\item The set of points $[\vorbit]\in\EOrbits^1_\vY/S^1$ where $\Fmap_\vorbit$ has $+1$ or $-1$ as an eigenvalue is a $0$-submanifold of $\OModuli^1_\vY$. Moreover, for every such  $\vorbit$, the total multiplicity of these two eigenvalues is $1$.
\item If $-1$ is an eigenvalue of $\Fmap_\vorbit$ for $\vorbit\in\EOrbits^1_\vY$, and $\xi$ is a generator for $\Ker(\Fmap_\vorbit+Id)$, then the vectors $d_r\Fmap_{\vorbit}$ and $d^2_\xi\Fmap_{\vorbit}$ are non-trivial in  $\Coker(\Fmap_\vorbit+Id)$.   
\end{itemize}
\end{thm}
\begin{proof}
Building on the discussion preceding the statement of Theorem~\ref{thm:super-nice-paths}, one may use an argument similar to the one used in the proof of Theorem~\ref{thm:super-nice-vector-fields} to show that the intersection of the section $d_r\Fmap_{Y_0,Y_1}$ of the line bundle 
$\Ccal_{Y_0,Y_1}\ra \Hcal_{Y_0,Y_1}$ 
with the zero section is transverse, and the zero locus $\Gcal^{\jmath}_{Y_0,Y_1}$ of $d_r\Fmap_{Y_0,Y_1}$ is thus a Banach manifold which fibers over $\Ycal_{Y_0,Y_1}$ via a Fredholm projection map  $p^\jmath_{Y_0,Y_1}$ of index $0$. For $\vY$ in the set $\Ycal^\jmath_{Y_0,Y_1}$ of regular values of $p^\jmath_{Y_0,Y_1}$ and $\vorbit\in \Hcal_\vY$, it thus follows that $d_r\Fmap_\vorbit$ is non-trivial in  $\Coker(\Fmap_\vorbit+Id)$. Similarly, the intersection of the bundle homomorphism 
\[d_x^2\Fmap_{Y_0,Y_1}:\Kcal_{Y_0,Y_1}\ra\Ccal_{Y_0,Y_1}\]
 with the zero homomorphism  is transverse in the complement of the zero section of $\Kcal_{Y_0,Y_1}$. Therefore, the zero locus $\Gcal^{\kappa}_{Y_0,Y_1}$ of $d^2_x\Fmap_{Y_0,Y_1}$ is a Banach manifold which fibers over $\Ycal_{Y_0,Y_1}$ via a Fredholm projection map $p^\kappa_{Y_0,Y_1}$ of index $1$. Once again, for $\vY$ in the set $\Ycal^\kappa_{Y_0,Y_1}$ of regular values of $p^\kappa_{Y_0,Y_1}$, $\vorbit\in \Hcal_\vY$, and non-zero $\xi\in\Ker(\Fmap_\vorbit+Id)$, the vector $d^2_{\xi}\Fmap_\vorbit$ is non-trivial in $\Coker(\Fmap_\vorbit+Id)$. With the above discussion  in place, in order to prove Theorem~\ref{thm:super-nice-paths} it suffices to set
\[\Ycal^{\supernice}_{Y_0,Y_1}=\bigcap_{\lambda\in\Lambda\cup\{\imath,\jmath,\kappa\}}\Ycal^\lambda_{Y_0,Y_1}.\]
For every super-nice path $\vY\in\Ycal^{\supernice}_{Y_0,Y_1}$, the first and the second conditions are satisfied since $\vY\in\Ycal ^\infty_{Y_0,Y_1}$. The third condition follows since $\vY\in\Ycal^\lambda_{Y_0,Y_1}$ for every root $\lambda\neq \pm 1$ of unity and the fourth condition follows since 
\[\vY\in \Ycal^{+1}_{Y_0,Y_1}\cap \Ycal^{-1}_{Y_0,Y_1}\cap \Ycal^{\imath}_{Y_0,Y_1}.\] 
Finally, for the last condition, note that if $-1$ is an eigenvalue of $\Fmap_\vorbit+Id$, then by the previous parts, the kernel and the cokernel of $\Fmap_\vorbit+Id$ are both $1$-dimensional. If $\xi\neq 0$ generates the kernel of $\Fmap_\vorbit+Id$, it follows from $\vY\in \Ycal^{\jmath}_{Y_0,Y_1}$ that $d_r\Fmap_{\vorbit}=d_r\Fmap_{Y_0,Y_1}(\vorbit)$ is non-trivial in $\Coker(\Fmap_\vorbit+Id)$ and it follows from  $\vY\in\Ycal^\kappa_{Y_0,Y_1}$ that $d^2_\xi\Fmap_\vorbit=d^2_x\Fmap(\vorbit;\xi)$  is non-trivial in  $\Coker(\Fmap_\vorbit+Id)$, completing the proof.
\end{proof}
\subsection{A pair of instructive  examples}\label{subsec:examples}
In this subsection, we discuss a pair of examples. To simplify the notation, we use the interval $[-1,1]$ (instead of $[0,1]$) to parametrize the paths of vector fields in these two example. The first example illustrates why it is necessary to complete $\EOrbits/S^1$ by gluing $\COrbits$ to it, while the second example illustrates why contribution from the double cover of certain super-rigid embedded periodic orbits to the wight function is non-trivial. These examples are in fact discussed as a preface to the proof of the invariance result of the next section (i.e. Theorem~\ref{thm:invariance-count-function}). 
\begin{examp}[(Trading periodic orbits for ghost orbits)]\label{ex:1}
Let $M$ be the complex plane $\C$ (we forget the compactness assumption on $M$ for the purpose of this example). Assume that the paths $\vY^\pm$ of vector fields on $\C$ are given by 
\[\vY^\pm(t,z)=\left(-t+i\pm |z|^2\right)\cdot z\in\C,\quad\quad\forall\ t\in[-1,1],\ z\in\C.\]
If $z$ is a zero of $\vY^\pm_t$, the equation $\vY^\pm(t,z)=0$ implies that either $z=0$, or $-i=\pm|z|^2-t$. Since the latter is not possible, the origin is the only zero of the vector field $\vY^\pm_t$ for every $t\in[-1,1]$. The differential $d_0\vY^\pm_t$ is non-singular, and is in fact identified as multiplication by  $-t+i$. The latter linear map has $-t\pm i$ as its eigenvalues. Let $\corbit^\pm_t=(t,\vY^\pm,0,\C)\in\Ncal_{\vY^\pm}$. Then 
\[\Ncal_{\vY^\pm}=\left\{\corbit^\pm_t\ \big|\ t\in[-1,1]\right\}\simeq [-1,1]\quad\text{and}\quad\lambda(\corbit^\pm_t)=-t+i,\quad\forall\ t\in[-1,1].\]
Therefore, $\COrbits^1_{\vY^\pm}$ is identified with  $[0,1]\subset[-1,1]$. For $t\in[0,1]$ we have $n(\corbit^\pm_t)=-1$, from Definition~\ref{def:weight-constant}, and $s(\corbit^\pm_t)=2\pi$.  Suppose now that $\gamma$ is a closed $\vY^\pm_t$ orbit of period $s$. In  polar coordinates,  
\begin{align*}
\gamma(\theta)=r(\theta)\cdot e^{i\alpha(\theta)}\quad\quad\text{and}\quad\quad\quad&\dot\gamma=(\dot{r}+i\dot{\alpha})\cdot e^{i\alpha}=s\cdot \left(-t\pm r^2+i\right)\cdot e^{i\alpha}\\
\Leftrightarrow\quad&\dot\alpha=s\quad\text{and}\quad \dot{r}=s\cdot(-t\pm r^2).
\end{align*}
In particular, for $r$ to be a periodic function of $\theta\in\R$, we need to have $t=\pm r^2$. In other words, there are no non-constant closed $\vY^+$-orbits if $t\in[-1,0]$ and there are no non-constant closed $\vY^-$-orbits if $t\in[0,1]$. For $t\in[-1,1]\setminus\{0\}$, there is a unique embedded periodic orbit
\begin{align*}
\orbit^{\sgn(t)}_t=(2\pi,\vY^{\sgn(t)}_t,\gamma_t),
\quad\quad\text{where}\quad\gamma_t(\theta):=\sqrt{|t|}\cdot e^{2\pi\theta}.
\end{align*}
Here $\sgn(t)\in\{+,-\}$ is the sign of $t$. Thus, we have
\[\EOrbits^1_{\vY^+}/S^1=\big\{\orbit^+_t\ \big|\ t\in(0,1]\big\}\quad\text{and}\quad \EOrbits^1_{\vY^-}/S^1=\big\{\orbit^-_t\ \big|\ t\in[-1,0)\big\}.\]
Therefore,  $\OModuli^1_{\vY^+}$ and $\OModuli^1_{\vY^-}$ are both identified with the interval $[-1,1]$. Under these identifications, we have  
\[\pi_{\vY^\pm}:\OModuli^1_{\vY^\pm}\simeq[-1,1]\ra[-1,1],\quad\quad\quad\pi_{\vY^{+}}(t)=|t|\quad\text{and}\quad \pi_{\vY^{-}}(t)=t,\quad\forall\ t\in[-1,1].\] 
 Note that the $\pi_{\vY^+}$ is smooth on $\OModuli^1_{\vY^+}\setminus\Bcal^1_{\vY^+}$, but only continuous at $\Bcal^1_{\vY^+}\subset \OModuli^1_{\vY^+}$. At $\gamma_t(0)$, the linearized holonomy map $\Fmap_{\orbit^\pm_t}:\R\ra \R$  may then be computed to give $ \Fmap_{\orbit^\pm_t}(r)=c^\pm_t\cdot r$, where $0<c^-_t<1$ and $c^+_t>1$. Therefore,  we have 
\[\epsilon_1(\orbit^+_t)=\epsilon_2(\orbit^+_t)=-1\quad\forall\ t\in(0,1]\quad\quad\text{and}\quad\quad\epsilon_1(\orbit^-_t)=\epsilon_2(\orbit^-_t)=-1\quad\forall\ t\in[-1,0),\ d>1.\] 
These computations imply that $\OModuli^1_{\vY^\pm_t}$ is finite, while
\[n(\OModuli^1_{\vY^-_t})=-1,\quad n(\OModuli^1_{\vY^+_t})=0\quad\text{and}\quad n(\OModuli^d_{\vY^\pm_t})=0,\quad\forall\ d>1,\] for every $t\in[-1,1]$.
\qed
\end{examp}

The next example illustrates how a sequence of embedded periodic orbits converge to the double cover of an embedded periodic orbit as we move the vector field in a generic path.

\begin{examp}[(The neighborhood of a degree-$2$ periodic orbit)]\label{ex:2}
 Let $\bump:\R\ra [0,1/100]$ denote a smooth bump function which is equal to $1/100$ on $[-20,20]$ and vanishes outside $[-40,40]$. We may further assume that $\bump(-x)=\bump(x)$ and that $\bump'(x)\leq 0$ for $x>0$. Define a family of maps 
\begin{align*}
f_t:\R^2\ra\R^2,\quad\quad f_t(x,y):=\left(-x+\bump(x)\left(x^2-xt\right),-2y\right)\quad\forall\ (x,y)\in \R^2,\ t\in[-1,1].
\end{align*}
It may then be easily checked that $f_t$ is a diffeomorphism for every $t\in[-1,1]$, which is close to the diffeomorphism $f:\R^2\ra\R^2$, defined by $f(x,y)=(-x,-2y)$. In fact, $sf+(1-s)f_t$ is a diffeomorphism for every $s\in[0,1]$.  Let $M_t$ denote the mapping cylinder of $f_t$, i.e.
\[M_t=\frac{\R\times\R^2}{\sim_t},\quad\quad (\theta,f_t(x,y))\sim_t(\theta+1,x,y).\]
The isotopy between $f_t$ and $f$ gives a diffeomorphism from $M_t$ to the mapping cone $M_f$ of $f$ (which in turn is diffeomorphic to  $(\R/\Z)\times \R^2$). Under this identification, the vector field $\partial_\theta$ gives a vector field $Y_t$ on $M_f$, and in fact, a path $\vY$ of vector fields with $\vY(t,\cdot)=Y_t$. The periodic orbits of $Y_t$ are then identified with the orbits of the periodic points of $f_t$.  Since the origin is the unique fixed point of each $f_t$, it determines a unique closed $Y_t$ orbit $\gamma_0$ of period $1$, which is given by $\gamma_0(\theta)=(\theta,0,0)$. The holonomy map of $Y_t$ at $\gamma_0(0)=(0,0,0)\in M_f$ may then be identified with $f_t:\R^2\ra \R^2$. This determines a point $\orbit_t=(t,1,\vY,\gamma_0)\in\EOrbits^1_\vY$, while the linearized holonomy map of $\orbit_t$ is given by the differential of $f_t$ at $(0,0)$, i.e.
\begin{align*}
\Fmap_t=\Fmap_{\orbit_t}=\left(\begin{array}{cc}
-1-\frac{t}{100}&0\\ 0&-2
\end{array}\right).
\end{align*}
Therefore, $\epsilon_1(\orbit_t)=1$ while $\epsilon_2(\orbit_t)=\sgn(t)$. This implies that $n([\orbit_t])=1$ for all $t\in[-1,1]$, while $n([2\star\orbit_t])=0$ for $t>0$ and $n([2\star\orbit])=-1$ for $t<0$. The periodic orbit $\orbit_0$ is $1$-rigid, but not super-rigid.
Next, we consider closed $\vY$ orbits of period $2$. Such periodic orbits correspond to periodic points of $f_t$ of period $2$, i.e. solutions in
\begin{align*}
Z_t=\left\{(x,y)\in\R^2\ |\ f_t(f_t(x,y))=(x,y)\right\}&=\Big\{(x,0)\ \big|\ x=0\ \ \text{or}\ \  \bump(x)=0 \ \ \text{or}\ \  x(x-t)\bump(x)=2t\Big\}.
\end{align*}
In particular, the intersection of $Z_t$ with $U=(-20,20)\times \R\subset\R^2$ is given by 
\begin{align*}
Z_t\cap U=\{(0,0)\}\cup \Big\{(x,0)\in\R^2\ \big|\ x^2-xt-200t=0\Big\}.
\end{align*}
If $t\leq 0$ then $Z_t\cap U=\{(0,0)\}$, while for $t>0$, $Z_t\cap U$ includes a pair of solutions
\[z^\pm_t=\left(\frac{t\pm\sqrt{t^2+800t}}{2},0\right)=\left(x^\pm_t,0\right)\in Z_t\cap U.\]
In fact, $z^\pm_t\in$ is in $(-15, 15)\times \{0\}\subset U$. Together, these two points determine a closed $Y_t$ orbit of period  $2$, which is denoted by $[\orbit'_t]\in\EOrbits^1_{Y_t}/S^1\subset\OModuli^1_{Y_t}$. As $t$ converges to $0$, $[\orbit'_t]$ converges to the double cover of $[\orbit_0]$. The holonomy map associated with $\orbit'_t$ is the map
\[f_t\circ f_t:(\R^2,z_t^+)\ra (\R^2,z_t^+).\]
Direct computation implies that
\begin{align*}
A_t=d_{z_t^+}(f_t\circ f_t)=\left(\begin{array}{cc}
1-\frac{t}{50}-\left(\frac{t}{100}\right)^2&0\\ 0&4
\end{array}\right)\quad\quad\forall\ t\in(0,1].
\end{align*} 
Therefore, $\epsilon_1(\orbit'_t)=\epsilon_2(\orbit'_t)=-1$ for all $ t\in(0,1]$. The periodic orbits in $\OModuli_\vY$ 
which have image in $(-10,10)\times\R$ form a  subset $\vGamma$ of $\OModuli^1_\vY$.  Let $\vGamma^d$ denote the subset of $\vGamma$ which consists of periodic orbits of period $d$ and $\Gamma^d_t$ denotes the intersection of $\OModuli_{Y_t}$ with $\vGamma^d$. Each $\Gamma^d_t$ is then compact and open in $\OModuli_\vY$. In fact,  $\Gamma^1_t=\{[\orbit_t]\}$ for all $t$, while 
\begin{align*}
\Gamma^2_t=\begin{cases}
\left\{[2\star\orbit_t],[\orbit'_t]\right\}& \text{if}\ t>0\\
\left\{[2\star\orbit_t]\right\}&\text{if}\ t\geq 0
\end{cases}.
\end{align*}
Moreover, every compact and open subset of $\OModuli_\vY$ is a union of finitely many $\vGamma^d$s, and every compact and open subset of $\OModuli_{Y_t}$ is a union of finitely many $\Gamma^d_t$s for every  $t\in[-1,1]$. The above computations imply that $n(\Gamma^1_t)=1$ for all $t$, while $n(\Gamma^2_t)=-1$ for all non-zero $t\in[-1,1]$. Although the number of points in $\Gamma^2_t$ changes as $t$ passes from negative values to positive values, we see that the wight function $n(\Gamma^2_t)$ remains unchanged. For $d>2$, the above computations show that $n(\Gamma^d_t)=0$.
\end{examp}

\section{Invariance of the weight function}\label{sec:sign}
\subsection{The invariance of the weight function for generic vector fields}
The following theorem is the main technical result of the paper.
\begin{thm}\label{thm:invariance-count-function}
Fix $Y_0, Y_1\in\Ycal^{\supernice}$ and $\vY\in\Ycal^{\supernice}(Y_0, Y_1)$. If $\vGamma$ is a compact and open subset of $\OModuli_\vY$ and $\Gamma_s=\vGamma\cap\big(\{s\}\times \OModuli_{Y_s}\big)$ for $s\in[0,1]$, we have $n(\Gamma_0)=n(\Gamma_1)$.
\end{thm}
\begin{proof}
Note that $\OModuli^d_\vY\simeq \OModuli^1_\vY$ is a $1$-manifold for every $d\in\Z^+$. Therefore, $\vGamma$ is a union of finitely many components from  $\amalg_d\OModuli^d_\vY$.  Let us denote the projection map from $\OModuli^1_\vY$ to $[0,1]$ by $\pi_\vY$. There is a finite subset 
\[\mathfrak{t}=\left\{t_1,t_2,\ldots,t_N\right\}\subset [0,1],\quad \text{with}\quad  0<t_1<t_2<\cdots<t_N<1\] and the corresponding embeddings $\gamma_{j}\in\Xcal$ for $j=1,\ldots,N$ such that 
\[\big\{[\vorbit_j]=[t_j,s_j,\vY,\gamma_{j}]\big\}_{j=1}^{N}\] 
is the set of all orbits $[\vorbit]\in\EOrbits^1_\vY/S^1$ such that $\Fmap_\vorbit$ has $+1$ or $-1$ as an eigenvalue (c.f. Theorem~\ref{thm:super-nice-paths}). Moreover, there is a finite subset 
\[\mathfrak{t}'=\left\{t'_1,t'_2,\ldots,t'_{N'}\right\}\subset [0,1],\quad \text{with}\quad  0<t'_1<t'_2<\cdots<t'_{N'}<1\] and the corresponding points  
\[\big\{\vcorbit_j=(t'_j,s'_j,\vY,x_j,P_j)\big\}_{j=1}^{N'}\in\Bcal^1_{Y_0,Y_1},\] 
which are all the boundary orbits of $\vGamma\cap\OModuli^1_\vY$, i.e. the points in $\vGamma\cap\Bcal^1_\vY$. If $[\vorbit]\in\vGamma\cap\EOrbits_\vY/S^1$ and $t=t_\vorbit\in[0,1] \setminus(\mathfrak{t}\cup\mathfrak{t}')$, it follows that $\Psi_{d\star\vorbit}$ is surjective for all $d\in\Z^+$. Therefore, there is a map 
\[\tau:J=(t-\epsilon,t+\epsilon)\ra\OModuli_\vY\] 
with $\tau(s)=[\vorbit_s]$, such that $\vorbit_t=\vorbit$, $\pi_\vY\circ\tau=Id_J$ and the only points in a neighborhood of $[d\star\vorbit]$ in $\OModuli_\vY$ are the points $\{[d\star\vorbit_s]\ |\ s\in J\}$.  Moreover, $\epsilon(d\star\vorbit_s)$ remains fixed over $J$. On the other hand, if $\corbit\in\vGamma\cap\COrbits_\vY$, and $\pi_\vY(\corbit)=t$, there is a map
\[\tau:J=(t-\epsilon,t+\epsilon)\ra \OModuli_\vY\]
with $\tau(t)=\corbit=(t,s,\vY,x,P)$ such that $\pi_\vY\circ \tau$ is either $Id_J$, or one of the map $i^{\pm}$ sending $t\pm\delta$. If $\pi_\vY\circ \tau=Id_J$, then $n(\tau(s))$ remains constant along $J$. On the other hand, if $\pi_\vY\circ \tau=i^\pm$, the differential $d_x\vY(t)$ is degenerate, and $0$ is an eigenvalue of $d_x\vY(t)$ with multiplicity $1$. It follows that $n(\tau(t-\delta))=-n(\tau(t+\delta))$. From these two observations, it follows that   $n(\Gamma_s)$ remains constant over $[0,1]\setminus (\mathfrak{t}\cup\mathfrak{t}')$. \\

We may next fix $j\in\{1,\ldots,N\}$ and set $t=t_j$, $\gamma=\gamma_j$, $s=s_j$ and $\vorbit=\vorbit_j=(t,s,\vY,\gamma)$, which is either in $\EOrbits^{\Gamma_{1}}_{Y_0,Y_1}/S^1$ (case A) or in $\EOrbits^{\Gamma_{-1}}_{Y_0,Y_1}/S^1$ (case B). We will first study case B, which should be compared with Example~\ref{ex:2}, as we will see below. In this case, one may assume that $\gamma$ is an embedding. It also follows that $\Psi_\vorbit$ is surjective, and as before, there is a map 
\[\tau_1:J=(t-\epsilon,t+\epsilon)\ra\OModuli_\vY^1,\] 
represented by $s\mapsto \vorbit_s\in\EOrbits^1_\vY$, such that $\vorbit_t=\vorbit$, $\pi_\vY\circ\tau_1=Id_J$ and that the only intersection of $\vGamma$ with $\OModuli^1_Y$ in a neighborhood of $\vorbit$ is in the image of $\tau_1$. Moreover, by the discussion preceding Theorem~\ref{thm:super-nice-paths} we obtain a map 
\begin{align*}
&p_\vorbit:J\times \Ker(\Fmap_\vorbit+Id)=J\times\R\ra \R=\Coker(\Fmap_\vorbit+Id),\\
&p_\vorbit(r,x):=\gmap_\vorbit(r,\gmap_\vorbit(r,x))-x,\quad\quad\forall\ (r,x)\in J\times\R.
\end{align*}
Since $\gmap_\vorbit(r,0)=0$ for  $r\in J$, $(r,0)$ belongs to the zero locus of $p_\vorbit$. The zero set of $p_\vorbit$  gives a Kuranishi model for a neighborhood of $2\star \vorbit$ in $\OModuli_\vY$, in the sense that away from $(r,0)$ which corresponds to $\tau_1(r)$, every pair of points $\{(r,x), (r,\gmap_\vorbit(r,x))\}$ with $p_\vorbit(r,x)=0$ corresponds to an embedded orbit 
\[\tau_2(r,x)=\tau_2(r,\gmap_\vorbit(r,x))\in\OModuli^1_\vY\] 
in the neighborhood of $2\star \vorbit$. As observed in the discussion preceding the statement of Theorem~\ref{thm:super-nice-paths}  $p_\vorbit(r,x)=x q_\vorbit(r,x)$ for a function $q_\vorbit:J\times\R\ra \R$ with $q_\vorbit(t,0)=0$ and $\partial_xq_\vorbit(t,0)=0$. The last condition in the statement of Theorem~\ref{thm:super-nice-paths} implies that for $\vY\in\Ycal^{\supernice}_{Y_0,Y_1}$, we have
\begin{align*}
&0\neq \frac{\partial q_\vorbit}{\partial r}\big|_{(t,0)}=\frac{\partial^2 p_\vorbit}{\partial r\partial x}\big|_{(t,0)}=-\frac{\partial^2\gmap_\vorbit}{\partial r\partial x}\big|_{(t,0)}\quad\text{and}\\ & 0\neq \frac{\partial^2 q_\vorbit}{\partial x^2}\big|_{(t,0)}=\frac{1}{3}\frac{\partial^3 p_\vorbit}{\partial x^3}\big|_{(t,0)}=-2\frac{\partial^3\gmap_\vorbit}{\partial x^3}\big|_{(t,0)}-3\left(\frac{\partial^2 \gmap_\vorbit}{\partial x^2}\big|_{(t,0)}\right)^2.
\end{align*}
After applying a diffeomorphism, we may assume that the map $\gmap_\vorbit:J\times\R\ra\R $ is given by
\[\gmap_\vorbit(r,x):=x(-1+a(r-t)+bx+cx^2),\quad\forall\ (r,x)\in J\times\R,\] where $a,c+b^2\neq 0$. Thus, for $(r,x)$ in  a neighborhood of $(t,0)$ we have $p_\vorbit(r,x)=0$ if and only if 
\[r=-\frac{2(c+b^2)}{a}\cdot x^2+\epsilon_\vorbit(x),\quad\text{where}\quad \lim_{x\ra 0}\frac{\epsilon_\vorbit(x)}{x^2}=0.\]
 For simplicity, let us assume that $a,c+b^2>0$. Then for $r>t$ there are no zeros $(r,x)$ of $p_\vorbit$, while for $r\in (t-\epsilon,t)$, if $x_1(r),x_2(r)$ are the two roots of $p_\vorbit(r,x)=0$ we have 
 \[\gmap_\vorbit(r,x_1(r))=x_2(r),\quad \gmap_\vorbit(r,x_2(r))=x_1(r)\quad\text{and}\quad\tau_2(r,x_1(r))=\tau_2(r,x_2(r)).\]
We may thus set $\tau_2(r)$ equal to this common value to obtain a map 
\[\tau_2:I=(t-\epsilon,t)\ra \OModuli^1_\vY,\quad\text{with}\quad\pi_\vY\circ\tau_2=Id_{I}.\] 
It also follows that a neighborhood of $\vorbit$ in $\OModuli_\vY$ is identified with 
\[\big\{2\star\tau_1(r) \big|\  r\in J\big\}\cup\big\{\tau_2(r)\ \big|\ r\in I\big\}.\] 
The cases where $a$ or $c+b^2$ (or both) are negative are handled in a completely similar way.\\

As before, let us assume that in appropriate coordinates, $\Fmap_\vorbit$ is given by the matrix of Equation~\ref{eq:-holonomy-map}. Then for $r$ sufficiently close to $t$, $\epsilon_d(\tau_1(r))$ and $\epsilon_d(\tau_2(r))$ are given by   
\begin{align*}
&\epsilon_d(\tau_1(r))=\epsilon(Q^d)\cdot \sgn\Big(\Big(\frac{\partial \gmap_\vorbit}{\partial x}\big|_{(r,0)}\Big)^d-1\Big)=\begin{cases}-\epsilon(Q^d)&\text{if $d$ is odd}\\-\epsilon(Q^d)\cdot\sgn(a(r-t))&\text{if $d$ is even}\end{cases},\quad\text{and}\\
&\epsilon_d(\tau_2(r))=\epsilon(Q^{2d})\cdot \sgn\Big(\Big(1+\frac{\partial p_\vorbit}{\partial x}\big|_{(r,0)}\Big)^d-1\Big)=\epsilon(Q^{2d})\cdot\sgn\Big(\frac{\partial p_\vorbit}{\partial x}\big|_{(r,0)}\Big)=\epsilon(Q^{2d})\cdot\sgn(a(r-t))
\end{align*}   

We may thus summarize the above discussion as follows. For $[\vorbit]\in\vGamma\cap\OModuli^1_\vY$ in case B,  there is an interval $I$, which is $(t-\epsilon,t)$ or $(t,t+\epsilon)$, and  $C^1$ maps $\tau_1$ and $\tau_2$ from $J=(t-\epsilon,t+\epsilon)$ and $I$ (respectively) to $\OModuli^1_\vY$, which give a local model for $\OModuli_\vY$ via the following properties:

\begin{itemize}
\item $\pi_\vY\circ \tau_1$ and $\pi_\vY\circ \tau_2$ are the identity maps of $J$ and $I$, respectively.
\item $\tau_1(t)=\vorbit$ and $2\star\vorbit$ is the limit of $\tau_2(r)$ as $r$ approaches $t$. Moreover, $2\star\vorbit$ is not the limit of any sequence in $\OModuli_\vY\setminus \big(\tau_2(I)\cup2\star\tau_1(J)\big)$.
\item $\epsilon_d(\tau_1(r))$ remains constant for $r\in I$ if $d$ is odd, while for even values of $d$, its sign changes as $r$ passes $t$. For $r\in I$, we have $\epsilon_d(\tau_2(r))=-\epsilon_{2d}(\tau_1(r))$. 
\end{itemize} 

\begin{figure}
\def\svgwidth{12cm}
\begin{center}
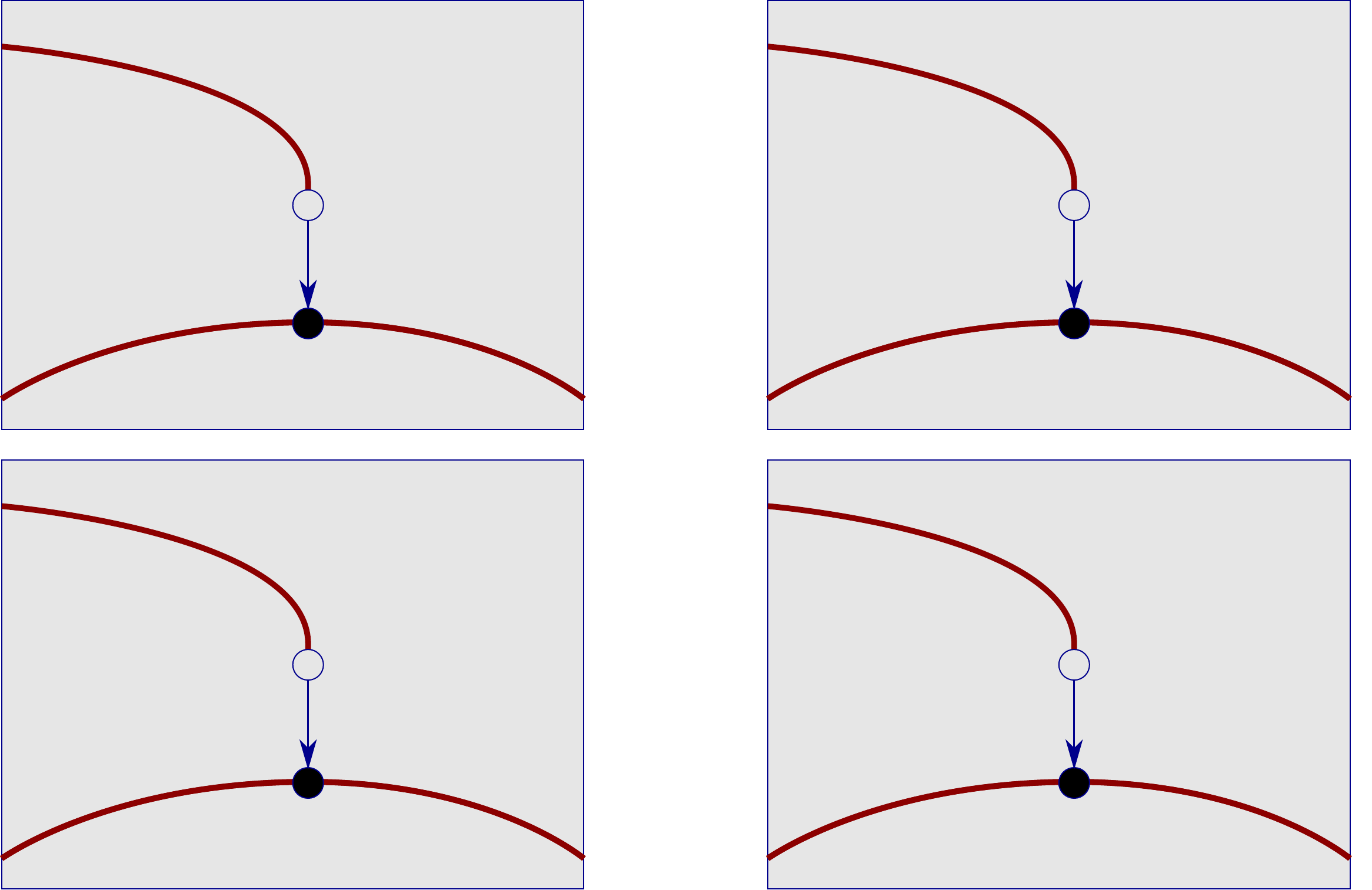
 \caption{\label{fig:local-contribution}
{A sequence of periodic orbits in $\EOrbits_\vY^1$ may converge to the double cover of a periodic orbit $\vorbit\in\EOrbits_\vY^1$. The maps $\tau_1$ and $\tau_2$ give a model for $\EOrbits_\vY$ near the limit orbit. The pair  $(\epsilon_1,\epsilon_2)$ associated with the periodic orbits in a neighborhood $\tau_1(t-\epsilon,t+\epsilon)$ of $\vorbit$ in  $\EOrbits_\vY^1$ and the periodic orbits in $\tau_2(t-\epsilon,t)\subset \EOrbits_\vY^1$ which are close to $2\star\vorbit$ follows one of the $4$ illustrated patterns. }}
\end{center}
\end{figure} 
 
The possible values for the function $(\epsilon_1,\epsilon_2)$  over the above local model is illustrated in Figure~\ref{fig:local-contribution} for $I=(t-\epsilon,t)$. In particular, the contribution of the aforementioned local model to $n(\Gamma_r)$ remains constant for $r\in J$. The case where $I=(t,t+\epsilon)$ is handled in a completely similar manner. The details of the argument in case B are best understood if the reader follows the steps for the vector field discussed in Example~\ref{ex:2}, as mentioned earlier.\\

In case A, $t$ is a critical value for $\pi_\vY$. There is thus a an interval $I$ which is either $(t-\epsilon,t]$ or $[t,t+\epsilon)$, and the maps  $\tau^{\pm}:I\ra \OModuli^1_\vY$ satisfying the following conditions:
\begin{itemize}
\item $\pi_\vY\circ\tau^+$ and $\pi_\vY\circ\tau^-$ are the identity map of $I$.
\item $\tau^+(t)=\tau^-(t)=[\vorbit]$ is a periodic orbit which is not a limit point of $\OModuli_\vg\setminus (\tau^+(I)\cup \tau^-(I))$.
\item $\Ker(\Fmap_{\tau^\pm(r)}-\lambda Id)$ is trivial for $r\in I^\circ$ and any root of unity $\lambda$, while $\Ker(\Fmap_{\vorbit}-Id)$ is $1$-dimensional and $\Ker(\Fmap_\vorbit-\lambda Id)$ is trivial for any other root of unity $\lambda$.
\item $\epsilon_d(\tau^+(r))=-\epsilon_d(\tau^-(r))$ for $d=1,2$.
\end{itemize}

Once again, the above observations imply that  $n(\Gamma_r)$ remains constant for $r$ belonging to an interval $(t-\epsilon,t+\epsilon)$ around $t$. Therefore, $n(\Gamma_t)$ is constant on $[0,1]\setminus\mathfrak{t}'$.\\

To complete the proof, we also need to show that $n(\Gamma_t)$ remains constant as $t$ passes the values $t_1',\ldots,t'_{N'}$. Let $\vcorbit=(t,s,\vY,x,P)=\vcorbit_j$ be a point in $\Bcal^1\cap\OModuli^1_\vY$ for some $j=1,\ldots,N'$. Note that $\OModuli_\vY$  is a $1$-manifold in a neighborhood $\pi_\vY^{-1}(t)$. Therefore, a neighborhood of $\vcorbit$ in $\OModuli^1_\vY$ may be identified as the image of an embedding 
\[\tau:(t-\epsilon,t+\epsilon)\ra \OModuli^1_\vY,\quad\quad\text{with}\quad\tau(t)=\vcorbit,\ \ \tau(t-\epsilon,t]\subset\EOrbits^1_\vY/S^1\ \ \text{and}\ \ \tau(t,t+\epsilon)\subset\COrbits^1_\vY.\] 
Moreover, we may assume that the multiple covers of the image of $\tau$ are the only points in $\OModuli_\vY$ which are close to multiple covers of $\vcorbit$.  The composition of $\tau$ with the projection map $\pi_\vY:\OModuli^1_\vY\ra [-1,1]$ may then be identified as $t+\delta\mapsto t\pm \delta$ or $t+\delta\mapsto t\pm |\delta|$ for $\delta\in(-\epsilon,\epsilon)$, c.f. Example~\ref{ex:1}. Let us assume that 
\[\pi_\vY\circ \tau=Id:(t-\epsilon,t+\epsilon)\ra (t-\epsilon,t+\epsilon).\]
Other cases are handled in a completely similar manner. In this case, it suffices to show that 
\[\epsilon_1(\tau(t-\delta))=\epsilon_2(\tau(t-\delta))=n(\tau(t+\delta))\]
for sufficiently small real numbers $\delta>0$. However, from the identification of a neighborhood of $\vcorbit\in\Bcal^1_{Y_0,Y_1} $
in $\EOrbits^1_{Y_0,Y_1}/S^1$ with $\Bcal^1_{Y_0,Y_1}\times\R^+$ (see the proof of Theorem~\ref{thm:gluing-moduli-spaces}) it follows that every point in $\EOrbits^1_\vY/S^1$ which is sufficiently close to $\vcorbit$ is definite. More precisely, $d_xY_t$ has an eigenvalue on $i\R$ with multiplicity $1$. Further, we may assume that in appropriate coordinates
\[d_xY_t=\left(\begin{array}{ccc}
0&-a&0\\a&0&0\\ 0&0&Q\end{array}\right),\]
where $a$ is a non-zero real number and $Q$ is a $(n-2)\times (n-2)$ matrix without any eigenvalues on $i\R$. The linearized holonomy map associated with $\tau(t-\delta)$, for sufficiently small values of $\delta>0$, is then close to a matrix of the form
\[\Fmap'_{\tau(t-\delta)}=\left(\begin{array}{cc}
\lambda_\delta&0\\0&\exp\left(\frac{s}{2\pi}Q\right)\end{array}\right),\]
where $\lambda_\delta$ is a positive real number not equal to $1$. As in Example~\ref{ex:1}, the assumption on $\pi_\vY\circ \tau$ implies (similar to case of $\vY^-$ in the aforementioned example) that $0<\lambda_\delta<1$. Therefore, 
\[\epsilon_1(\tau(t-\delta))=\epsilon_2(\tau(t-\delta))=-\sgn\left(\det\left(\exp(Q)\right)\right)=-(-1)^{c(Q)}=(-1)^{c_-(\tau(t+\delta))},\]
where $c(Q)$ denotes the number of eigenvalues $\lambda$ of $Q$ (counting with multiplicity) with $\mathrm{Re}(\lambda)<0$. This completes the proof of the claim when $\pi_\vY\circ\tau$ is the identity map. The proof in the other cases is similar.   
\end{proof}
\subsection{The weight function for an arbitrary vector field}
\label{subsec:count-function-general}
Given $Y\in\Ycal$, a subset $\Gamma$ of $\OModuli_Y$ is called {\emph{isolated}} if $\Gamma$ is compact and open in $\OModuli_Y$. For such $\Gamma$, there are bounded open subsets $\Ucal=\Ucal_{\Gamma}$ and $\Ucal'=\Ucal'_\Gamma$ of $\OModuli$ with  $\overline{\Ucal}\subset\Ucal'$ and $\Gamma=\OModuli_Y\cap\Ucal= \OModuli_Y\cap\Ucal'$. If there is a sequence $\{Y_j\}_j$ in $\Ycal$ which converges to $Y$ such that $\Gamma_j= \OModuli_{Y_j}\cap\Ucal $ is not closed, then there are  $\{\corbit^k_j\}_k\in\OModuli_{Y_j}$ which converge to a point  $\corbit_j\in\Ucal\cap\OModuli_{Y_j}$. By passing to a sub-sequence, we may assume that $\{\corbit_j\}_j$ converges to a point  $\corbit\in\overline{\Ucal}\cap\OModuli_Y=\Gamma$. If $k_j$ is sufficiently large, the sequence $\{\corbit_j^{k_j}\}_j$ (which is outside $\Ucal$) will also converge to $\corbit$. Since $\OModuli\setminus\Ucal$ is closed,  $\corbit$ is included in $\OModuli\setminus\Ucal$ and $\Gamma\cap(\OModuli\setminus\Ucal)$ is non-empty. This contradiction implies that there is an open and path connected neighborhood $U=U_{\Gamma}$ of $Y$ in $\Ycal$ such that for every  $Y'\in U$, the set $\Gamma_{Y'}=\OModuli_{Y'}\cap\Ucal$ is isolated (since $\Ucal$ is open and bounded). Note that the choice of $\Ucal_{\Gamma}$ and $U_{\Gamma}$ is not unique. 

\begin{defn}\label{defn:contribution}
If $\Gamma$ is an isolated subset of $\OModuli_Y$ for some $Y\in\Ycal$, define the {\emph{weight}} of $\Gamma$ to the {\emph{orbit count function}} by $n(\Gamma):=n(\Gamma')$, where $Y'$ is an arbitrary vector field in $\Ycal^{\supernice}\cap U_{\Gamma}$ and $\Gamma'$ is the finite set $\OModuli_{Y'}\cap\Ucal_{\Gamma}$. 
\end{defn}

\begin{thm}\label{thm:invariance}
If $Y\in\Ycal$ and $\Gamma$ is an isolated subset of $\OModuli_Y$, the weight $n(\Gamma)$ is independent of the choices made in Definition~\ref{defn:contribution}. Moreover, if $\vY\in\Ycal_{Y_0,Y_1}$ is a path connecting $Y_0$ to $Y_1$ and $\vGamma$ is an isolated  subset of $\OModuli_\vY$, then $n(\Gamma_0)=n(\Gamma_1)$, where $\Gamma_i=\vGamma\cap\OModuli_{Y_i}$ for $i=0,1$.
\end{thm}

\begin{proof}
If $Y_0,Y_1\in \Ycal^{\supernice}\cap U_{\Gamma}$ are a pair of super-nice vector fields, choose a super-nice path $\vY$ in $U_\Gamma$ which connects them. Then $\vGamma=\OModuli_\vY\cap([0,1]\times \Ucal_\Gamma)$ is open and compact. If $\Gamma_i=\Ucal_\Gamma\cap\OModuli_{Y_i}$ for $i=0,1$, Theorem~\ref{thm:invariance-count-function} implies $n(\Gamma_0)=n(\Gamma_1)$. Therefore,  $n(\Gamma)$ is independent from the choice of $Y'\in U_\Gamma$.  But this implies the independence from the choice of $U_{\Gamma}$ and $\Ucal_{\Gamma}$ as well, as we may always pass to the intersection of two different choices for these open sets.\\

By the above argument, for the second claim we may further assume that $Y_0$ and $Y_1$ are super-nice. Since $\vGamma$ is isolated, there are bounded open sets $\Ucal$ and $\Ucal'$ in $[0,1]\times \Xcal$ so that 
\[\overline{\Ucal}\subset\Ucal'\quad\text{and}\quad \vGamma=\OModuli_\vY\cap\Ucal=\OModuli_\vY\cap\Ucal'.\]
As discussed after Definition~\ref{defn:contribution}, this implies that there is a neighborhood $U$ of $\vY$ in $\Ycal_{Y_0,Y_1}$ such that for every other path $\vY'\in U$, $\vGamma'=\OModuli_{\vY'}\cap\Ucal$ is isolated. If $\vY'$ is a super-nice path connecting $Y_0$ to $Y_1$, Theorem~\ref{thm:invariance-count-function} implies that $n(\Gamma_0)=n(\Gamma_1)$, completing the proof.
\end{proof}

For some of the vector fields which are faced in interesting dynamical systems, the families of periodic orbits are not finite and isolated, but have the structure of closed manifolds. Following the argument of \cite[Theorem 5.1]{White-1}, one can prove the following proposition.

\begin{prop}\label{prop:computation}
Suppose that $M$ is a smooth closed manifold of dimension $n$ as before, $Y\in\Ycal$ and $\Gamma$ is a compact and open subset of $\EOrbits^1_Y/S^1$, which has the structure of a closed $k$-dimensional manifold. Moreover, suppose that for every $[\orbit]\in\Gamma$, we have 
\[\det(\hslash \cdot Id-\Fmap_\orbit)=(\hslash-1)^k\qfrak\in\R[\hslash]\] 
where  $\qfrak=\qfrak_\orbit$ does not vanish at any root of unity. If $m_1(\qfrak)$ and $m_2(\qfrak)$ denote the number of real roots of $\qfrak$ in $(-\infty,1)$ and $(-1,1)$, respectively, we have  
\begin{align*}
n(\Gamma)=(-1)^{m_1(\qfrak)}\cdot \chi(\Gamma),\quad n(2\star \Gamma)=\frac{(-1)^{m_2(\qfrak)}-(-1)^{m_1(\qfrak)}}{2}\cdot \chi(\Gamma)\quad\text{and}\quad n(d\star\Gamma)=0\quad\forall\ d>2,
\end{align*}
where $\chi(\Gamma)$ denotes the Euler characteristic of $\Gamma$. 
\end{prop}

\subsection{Examples and computations}\label{subsec:holonomy}
In this subsection, we will consider a few examples and compute the weight function for some compact and open subsets of the moduli space  of ghost and periodic orbits.
\begin{examp}[{\bf{(The mapping cylinder of a diffeomorphism)}}]
 Choose a closed smooth manifold $N$ of dimension  $\dim(N)$ and let $f:N\ra N$ denote a diffeomorphism from $N$ to itself. Let $M=C(f)$ denote the mapping cylinder of $f$. The vector field $\partial_\theta$ on $\R\times N$ induces a vector field $Y_f$ on $M$ without fixed points. Closed $Y_f$ orbits of length $n$ correspond to elements of the set $\Periodics_n(f)$ of periodic points of $f$ of period $n$. Associated with every $p\in\Periodics_n(f)$, the $n$-point orbit 
\[\Ocal(p)=\Ocal_f(p)=\{p,f(p),f^2(p),\ldots,f^{n-1}(p)\}\subset \Periodics_n(f)\]
of $p$ is invariant under the flow and gives a periodic orbit $\orbit_p$ of length $n$, which starts at $(0,p)$. Note that $[\orbit_{f^j(p)}]=[\orbit_p]$ for $j=1,\ldots,n-1$. The holonomy map of $\orbit_p$ is identified with the germ of $f^n$ at $p$. Assume further that for every  $p\in\Periodics_n(f)$  as above, $d_pf:T_pN\ra T_pN$ does not have any eigenvalues which is a root of unity. Then $\epsilon_d(\orbit_p)$ is the sign of $\det(d_pf^{nd}-Id)$, which is the index $\imath(f^{nd},p)$ of the fixed point $p$ of $f^{nd}$. The latter number is independent of the choice of $p$ in its orbit and may be denoted by $\imath(f^{nd},\Ocal(p))$. The set $\Gamma_n(f)$ of periodic orbits of length $n$ is finite and consists of  super-rigid periodic orbits. Let 
\[\pi_f=\sum_{n=1}^\infty n(\Gamma_n(f))\hslash^n\] 
denote the associated formal power series. For odd values of $n$, we have  
\[n(\Gamma_n(f))=\frac{1}{n}\cdot\sum_{p\in\Periodics_n(f)}\imath(f^n,p).\]
If $n$ is even, the orbit of every point $p\in\Periodics_{n}(f)$ contributes $\imath(f^{d},p)$ to $n(\Gamma_n(f))$, while every $p\in\Periodics_{n/2}(f)$ contributes $(\imath(f^{n},p)-\imath(f^{n/2},p))/2$ to $n(\Gamma_n(f))$. It follows that
\[n(\Gamma_n(f))=\frac{1}{n}\sum_{p\in\Periodics_{n}(f)\cup \Periodics_{n/2}(f)}\imath(f^{n},p)-\frac{1}{n}\sum_{p\in\Periodics_{n/2}(f)}\imath(f^{n/2},p).\]
For $p\in\Periodics_n(f)$ note that $\imath(f^{nk},p)$ only depends on the parity of $k$, by our previous observations. Given an integer $n$, let $n=2^r\cdot k$, where $k$ is odd. Correspondingly, we may compute
\begin{align*}
\sum_{d=2^il|n=2^rk}d\cdot n(\Gamma_d(f))&=\sum_{l|k}\sum_{j=0}^{r}\sum_{p\in\Periodics_{2^jl}(f)}\imath(f^{n},p)\quad\quad\Rightarrow\quad\quad \sum_{d|n}d\cdot n(\Gamma_d(f))=L(f^n),
\end{align*}
where $L(f^n)$ denotes the Lefshetz number of $f^n:N\ra N$. By M\"ubius inversion formula,  
\begin{align*}
n(\Gamma_d(f))=\frac{1}{d}\sum_{l|d}\mu\left(\frac{d}{l}\right)L(f^l).
\end{align*}
The map $f:N\ra N$ induces the linear maps on homology $\Fmap_j:\Ht_j(N,\Q)\ra \Ht_j(N,\Q)$, which have eigenvalues $\lambda_j^1,\ldots,\lambda_j^{b_j(N)}$, where $b_j(N)$ denotes the $j$-th Bettin number of $N$. Thus
\begin{equation}\label{eq:pi-orbit}
\pi_f=\sum_{k=0}^{\dim(N)}\sum_{j=1}^{b_k(N)}(-1)^k\Pmap_{\lambda_k^j}(\hslash),\quad\text{where}\quad\Pmap_\lambda(\hslash):=\sum_{d=1}^\infty\sum_{l|d}\mu\left(\frac{d}{l}\right)\frac{\lambda^l\hslash^d}{d}
\end{equation}
\qed
\end{examp}
\begin{examp}[{\bf{(Vector fields on $S^3$)}}]
We think of $S^3$ as the subspace of $\C^2$ defined by 
\[S^3=\big\{(x,y)\in\C^2\ \big|\ |x|^2+|y|^2=1\big\}.\]
The tangent space at $(x,y)$ to the sphere is then given as the subspace of $\C^2=T_{(x,y)}\C^2$ given by the pairs $(z,w)\in\C^2$ with $\mathrm{Re}(\bar x\cdot z+\bar y\cdot w)=0$. In particular, the Hopf vector field is given by $H(x,y)=(ix,iy)$. All orbits of $H$ are periodic of fixed period $2\pi$, and the space of fibers may be identified as $S^2=\mathbb{CP}^1$. Therefore, $\OModuli_H=\amalg_{d\in\Z^+} \Gamma_d$, where each $\Gamma_d$ is a copy of $S^2$ and consists of the degree-$d$ covers of the simple orbits of $H$. The vector field $H$ may be perturbed to the vector field 
\[H_{a,b}(x,y):=(iax,iby),\quad\quad\forall\ \ (x,y)\in S^3\subset\C^2,\]
where $a$ and $b$ are real numbers close to $1$. The orbit of $H_{a,b}$ which starts at $(x,y)\in S^3$ is then given by \[\gamma_{x,y}(\theta)=(xe^{ia\theta},ye^{ib\theta},\quad\quad\forall\ \ \theta\in\R.\]
If we chose $a,b\in\R$ so that they are independent over $\Q$, the orbit determined by $\gamma_{x,y}$ is not periodic unless $x=0$ or $y=0$. Correspondingly, we obtain a pair of closed orbits, denoted by $\orbit_a$ and $\orbit_b$ with periods $2\pi/a$ and $2\pi/b$ (respectively) which pass through $(1,0)$ and $(0,1)$ (respectively) at time $\theta=0$. The above computation of all orbits of $H_{a,b}$ may then be used to compute the linearization of the holonomy maps associated with $\orbit_a$ and $\orbit_b$ to be
\[\Fmap_{\orbit_a},\Fmap_{\orbit_b}:\C\ra\C,\quad\quad
\Fmap_{\orbit_a}(x)=e^{2\pi i\frac{b}{a}}\quad\text{and}\quad
\Fmap_{\orbit_b}(x)=e^{2\pi i\frac{a}{b}},\quad\forall\ x\in\C.\]
Therefore, both periodic orbits are definite and $\epsilon_d(\orbit_a)=\epsilon_d(\orbit_b)=1$ for all $d\in\Z^+$.
This computation quickly implies that $n(\Gamma_1)=2$ and $n(\Gamma_d)=0$ for $d>1$.\\
Alternatively, one may use Proposition~\ref{prop:computation} for a quick computation of $n(\Gamma_d)=n(d\star\Gamma_1)$, which matches with the above direct computation. 
\qed
\end{examp}
\begin{examp}[{\bf{(Volume preserving vector fields on $3$-manifolds)}}]
Let $M$ be a smooth closed $3$-manifold, which is equipped with a volume form $\Omega$ and $\Ycal^{\mathrm{v}}\subset\Ycal$ denote the subset of volume preserving vector fields. One can show that the intersection $\Ycal^{\mathrm{v},\supernice}=\Ycal^{\mathrm{v}}\cap\Ycal^\supernice$ is a Bair subset of $\Ycal^{\mathrm{v}}$. Given $Y\in \Ycal^{\mathrm{v}}$, and $\corbit=(s,Y,x,P)\in\COrbits_\vY$, note that $\trace(d_xY)=0$. This follows since $Y$ is volume preserving. Moreover, since $\tfrak(\corbit)<0$, it follows that $d_xY$ has three eigenvalues $\lambda,\bar\lambda$ and $\mu$, with $\lambda+\bar\lambda<0$ and $\mu\in \R^-$. In particular, $\sgn(\det(d_xY))=-1$. This implies that the total contribution of ghost periodic orbits to the weight function is non-negative for every volume preserving vector field $Y\in \Ycal^{\mathrm{v}}$ and every compact and open subset $\Gamma$ of $\OModuli_Y$.
\qed
\end{examp}
\begin{examp}[{\bf{(The geodesic flow)}}]
Let $M$ denote the unit tangent bundle of a smooth closed manifold $N$ of dimension $n$ and $\Ycal^{\mathrm{r}}\subset \Ycal$  consist of the vector fields generating the geodesic flows associated with the $C^2$ Riemannian metrics on $N$. The counting problem for closed geodesics was addressed in the earlier work \cite{Ef-p} of the author. One can check that
$\Ycal^{\mathrm{r},\supernice}=\Ycal^{\mathrm{r}}\cap\Ycal^{\supernice}\subset \Ycal^{\mathrm{r}}$
is a Bair subset and correspond to the Riemannian metrics which are super nice or {\emph{bumpy}} (c.f. \cite{Ef-p,Abraham,White-1}). For every $Y\in \Ycal^{\mathrm{r}}$, $\oCOrbits_Y$ is empty, since $Y$ admits no zeros. If $g$ is a negatively curved Riemannian metric, the closed orbits, which are the periodic orbits of $Y_g$ form a discrete subset of $\OModuli$. The linearization $\Fmap_\orbit$ of the holonomy map associated with every periodic orbit $\orbit$ is a hyperbolic map $\Fmap:\R^{2(n-1)}\ra\R^{2(n-1)}$ which has $n-1$ eigenvalues in $(0,1)$ and $n-1$ eigenvalues in $(1,\infty)$. It follows that all such periodic orbits are definite, and $\epsilon_1(\orbit)=\epsilon_2(\orbit)=(-1)^{n-1}$. In particular, our convention for the normalization of the weight function in this case differs from the convention set in \cite{Ef-p} by a factor of $(-1)^{n-1}$. One may check that this is the case for all compact and open subset of the space of closed geodesics (for arbitrary Riemannian metrics on $N$).  
\qed
\end{examp}

\end{document}

%% file: commands-file-nmd.tex
\newcommand{\qfrak}{\mathfrak{q}}
\newcommand{\bump}{\eta}
\newcommand{\ocorbit}{\overline\corbit}
\newcommand{\supernice}{\star}
\newcommand{\tracet}{\mathfrak{t}}
\newcommand{\oXcal}{{\Xcal}^\bullet}
\newcommand{\oZcal}{\Zcal^\bullet}

\newcommand{\vGamma}{\overline{\Gamma}}

\newcommand{\tfrak}{\mathfrak{t}}
\newcommand{\efrak}{\mathfrak{e}}
\newcommand{\corbit}{\mathfrak{y}}
\newcommand{\vcorbit}{\overline{\corbit}}
\newcommand{\fmapp}{f}

\newcommand{\vfmapp}{F}
\newcommand{\vgmapp}{G}
\newcommand{\vhmapp}{H}

\newcommand{\OModuli}{\mathcal{M}}
\newcommand{\COrbits}{\mathcal{Q}}
\newcommand{\EOrbits}{\mathcal{P}}
\newcommand{\bulletEOrbits}{\mathcal{P}^\bullet}
\newcommand{\oCOrbits}{\overline\COrbits}
\newcommand{\oEOrbits}{\overline\EOrbits}
\newcommand{\Periodics}{\mathscr{P}}

\newcommand{\sgn}{\mathrm{sgn}}

\newcommand{\vg}{{\overline{Y}}}
\newcommand{\vY}{{\overline{Y}}}
\newcommand{\vF}{{\overline{F}}}
\newcommand{\vPhi}{{\overline{\Phi}}}

\newcommand{\dgam}{{\dot{\gamma}}}

\newcommand{\xfrak}{\mathfrak{x}}

\newcommand{\fmap}{{{f}}}

\newcommand{\Fmap}{{\mathfrak{f}}}

\newcommand{\gmap}{{{g}}}

\newcommand{\gfrak}{{\mathfrak{g}}}
\newcommand{\hmap}{{{h}}}

\newcommand{\Pmap}{{\mathfrak{p}}}

\newcommand{\orbit}{{\mathfrak{z}}}
\newcommand{\vorbit}{{\overline{\mathfrak{z}}}}

\newcommand{\colvec}[2][.8]{%
  \scalebox{#1}{%
    \renewcommand{\arraystretch}{.8}%
    $\begin{pmatrix}#2\end{pmatrix}$%
  }
}

\newcommand{\trace}{\mathrm{tr}}

\newcommand{\x}{\mathrm{\bf{x}}}

\newcommand{\ra}{\rightarrow}

\newcommand{\Ht}{\mathrm{H}}

\newcommand{\Hbb}{\mathbb{H}}

\newcommand{\Bcal}{\mathcal{B}}
\newcommand{\Ccal}{\mathcal{C}}

\newcommand{\Ecal}{\mathcal{E}}
\newcommand{\Fcal}{\mathcal{F}}
\newcommand{\Gcal}{\mathcal{G}}
\newcommand{\Hcal}{\mathcal{H}}

\newcommand{\Kcal}{\mathcal{K}}

\newcommand{\Ncal}{\mathcal{N}}
\newcommand{\Ocal}{\mathcal{O}}

\newcommand{\Ucal}{\mathcal{U}}

\newcommand{\Xcal}{\mathcal{X}}
\newcommand{\Ycal}{\mathcal{Y}}
\newcommand{\Zcal}{\mathcal{Z}}

\newcommand{\Ker}{\mathrm{Ker}}

\newcommand{\Coker}{\mathrm{Coker}}

\newcommand{\Image}{\mathrm{Im}}

\newtheorem{thm}{Theorem}[section]
\newtheorem{prop}[thm]{Proposition}

\newtheorem{lem}[thm]{Lemma}

\theoremstyle{definition}
\newtheorem{examp}[thm]{Example}
\newtheorem{defn}[thm]{Definition}

\newtheorem{rmk}[thm]{Remark}



%% file: local-contribution.pdf_tex
\begingroup%
  \makeatletter%
  \providecommand\color[2][]{%
    \errmessage{(Inkscape) Color is used for the text in Inkscape, but the package 'color.sty' is not loaded}%
    \renewcommand\color[2][]{}%
  }%
  \providecommand\transparent[1]{%
    \errmessage{(Inkscape) Transparency is used (non-zero) for the text in Inkscape, but the package 'transparent.sty' is not loaded}%
    \renewcommand\transparent[1]{}%
  }%
  \providecommand\rotatebox[2]{#2}%
  \ifx\svgwidth\undefined%
    \setlength{\unitlength}{671.47401379bp}%
    \ifx\svgscale\undefined%
      \relax%
    \else%
      \setlength{\unitlength}{\unitlength * \real{\svgscale}}%
    \fi%
  \else%
    \setlength{\unitlength}{\svgwidth}%
  \fi%
  \global\let\svgwidth\undefined%
  \global\let\svgscale\undefined%
  \makeatother%
  \begin{picture}(1,0.64866399)%
    \put(0,0){\includegraphics[width=\unitlength]{local-contribution.pdf}}%
    \put(0.1280844,0.35223365){\color[rgb]{0,0,0}\makebox(0,0)[lb]{\smash{$\mathrm{Im}(\tau_1)$}}}%
    \put(0.0436752,0.54769917){\color[rgb]{0,0,0}\makebox(0,0)[lb]{\smash{$(-1,-1)$}}}%
    \put(0.06776932,0.61583279){\color[rgb]{0,0,0}\makebox(0,0)[lb]{\smash{$\mathrm{Im}(\tau_2)$}}}%
    \put(0.05484468,0.42036735){\color[rgb]{0,0,0}\makebox(0,0)[lb]{\smash{$(+1,+1)$}}}%
    \put(0.28716937,0.42036735){\color[rgb]{0,0,0}\makebox(0,0)[lb]{\smash{$(+1,-1)$}}}%
    \put(0.60214809,0.54769917){\color[rgb]{0,0,0}\makebox(0,0)[lb]{\smash{$(+1,+1)$}}}%
    \put(0.61331734,0.42036735){\color[rgb]{0,0,0}\makebox(0,0)[lb]{\smash{$(+1,-1)$}}}%
    \put(0.84340865,0.42036735){\color[rgb]{0,0,0}\makebox(0,0)[lb]{\smash{$(+1,+1)$}}}%
    \put(0.0436752,0.21261549){\color[rgb]{0,0,0}\makebox(0,0)[lb]{\smash{$(+1,+1)$}}}%
    \put(0.05484468,0.08528346){\color[rgb]{0,0,0}\makebox(0,0)[lb]{\smash{$(-1,-1)$}}}%
    \put(0.28716937,0.08528346){\color[rgb]{0,0,0}\makebox(0,0)[lb]{\smash{$(-1,+1)$}}}%
    \put(0.60214809,0.21261549){\color[rgb]{0,0,0}\makebox(0,0)[lb]{\smash{$(-1,-1)$}}}%
    \put(0.61331734,0.08528346){\color[rgb]{0,0,0}\makebox(0,0)[lb]{\smash{$(-1,+1)$}}}%
    \put(0.84564251,0.08528346){\color[rgb]{0,0,0}\makebox(0,0)[lb]{\smash{$(-1,-1)$}}}%
    \put(0.69996048,0.35223365){\color[rgb]{0,0,0}\makebox(0,0)[lb]{\smash{$\mathrm{Im}(\tau_1)$}}}%
    \put(0.63964558,0.61583279){\color[rgb]{0,0,0}\makebox(0,0)[lb]{\smash{$\mathrm{Im}(\tau_2)$}}}%
    \put(0.69996048,0.0186393){\color[rgb]{0,0,0}\makebox(0,0)[lb]{\smash{$\mathrm{Im}(\tau_1)$}}}%
    \put(0.63964558,0.28223828){\color[rgb]{0,0,0}\makebox(0,0)[lb]{\smash{$\mathrm{Im}(\tau_2)$}}}%
    \put(0.1280843,0.0186393){\color[rgb]{0,0,0}\makebox(0,0)[lb]{\smash{$\mathrm{Im}(\tau_1)$}}}%
    \put(0.06776948,0.28223828){\color[rgb]{0,0,0}\makebox(0,0)[lb]{\smash{$\mathrm{Im}(\tau_2)$}}}%
  \end{picture}%
\endgroup%